\documentclass{article}
\usepackage{amsmath,amsfonts,amssymb,amsthm}
\usepackage{mathtools}
\usepackage{dsfont}
\usepackage{enumerate}
\usepackage{esint}
\allowdisplaybreaks[4]

\newtheorem{theorem}{Theorem}[section]
\newtheorem{proposition}[theorem]{Proposition}
\newtheorem{lemma}[theorem]{Lemma}
\newtheorem{corollary}[theorem]{Corollary}

\newtheorem{remark}{Remark}[section]

\usepackage[colorlinks=true,linkcolor=blue,citecolor=green]{hyperref}

\usepackage[margin=2cm]{geometry}
\newlength{\bibitemsep}
\newlength{\bibparskip}
\let\oldthebibliography\thebibliography
\renewcommand\thebibliography[1]{%
  \oldthebibliography{#1}%
  \setlength{\parskip}{\bibitemsep}%
  \setlength{\itemsep}{\bibparskip}%
}

\numberwithin{equation}{section}

\begin{document}

\title{\Large\bf Non-uniqueness of the transport equation at high spacial integrability}
\author{Jingpeng Wu\thanks{Corresponding author. School of Mathematics and Statistics, Huazhong University of Science and Technology; Institute of Artificial Intelligence, Huazhong University of Science and Technology, Wuhan, 430074, China (jpwu\_postdoc@hust.edu.cn).}\and Xianwen Zhang\thanks{School of Mathematics and Statistics, Huazhong University of Science and Technology, Wuhan, 430074, China (xwzhang@hust.edu.cn)}}
\date{}
\maketitle

\begin{abstract}
In this paper, we show the non-uniqueness of the weak solution in the class $\rho\in L^{s}_tL^p_x$ for the transport equation driven by a divergence-free vector field $\boldsymbol{u}\in L^{\tilde{s}}_tW^{1,q}_x\cap L_t^{s'}L_x^{p'}$ happens in the range $1/p+1/q>1-\frac{p-1}{4(p+1)p}$ with some $\tilde{s}>1$, as long as $1\le s<\infty$, $p>1$. As a corollary, $L^{\infty}$ in time of the density $\rho$ is critical in some sense for the uniqueness of weak solution. Our proof is based on the convex integration method developed in \cite{MS20,CL21}.


\noindent{\bf MR Subject Classification:} 35A02; 35D30; 35Q35.
\end{abstract}

\section{Introduction}\label{sec1}

This paper deals with the problem of (non-)uniqueness of weak solution to the linear transport equation on the torus $\mathbb{T}^d:=\mathbb{R}^d\setminus\mathbb{Z}^d$ written by
\begin{equation}\label{eq-TE}
\left\{\begin{split}
&\partial_t\rho+\boldsymbol{u}\cdot\nabla\rho=0,\\
&\rho\vert _{t=0}=\rho_0,
\end{split}\right.
\end{equation}
where $\rho\colon[0,T]\times\mathbb{T}^d\to\mathbb{R}$ is the unknown density, $\rho_0$ is the initial data and $\boldsymbol{u}\colon[0,T]\times\mathbb{T}^d\to\mathbb{R}^d$ is a given divergence-free vector field, i.e.~$\operatorname{div} \boldsymbol{u}=0$ in the sense of distribution. In this case, the weak solution to \eqref{eq-TE} is defined as
\begin{equation*}
\int_{\mathbb{T}^d}\rho_0\phi(0,x)\,dx=\int_0^T\int_{\mathbb{T}^d}\rho(\partial_t\phi+\boldsymbol{u}\cdot\nabla\phi)\,dx\,dt\,\text{ for all }\phi\in C_c^{\infty}([0,T)\times\mathbb{T}^d).
\end{equation*}
Another equivalent definition is (see e.g.~\cite[p276]{DL88KE})
\begin{equation*}
\int_0^T\int_{\mathbb{T}^d}\rho(\partial_t\phi+\boldsymbol{u}\cdot\nabla\phi)\,dx\,dt=0\,\text{ for all }\phi\in C_c^{\infty}((0,T)\times\mathbb{T}^d),
\end{equation*}
and $\rho$ is continuous in time with $\lim_{t\to0^+}\rho(t)=\rho_0$ in the distributional sense.

Since the transport equation is linear, the existence of weak solutions can be obtained by the method of regularization even for very rough vector fields, see \cite[Proposition II.1]{DL89ODE}. However, to obtain the uniqueness, more complicated discussions will be involved. 

\subsection{Background}

In the smooth setting, the method of characteristics solves first-order PDE by converting the PDE into an appropriate system of ODE \cite[Sec.3.2]{Eva10}. In particular, for Lipschitz vector fields $u\in C_tW^{1,\infty}_x$, the solution to \eqref{eq-TE} can be given by the Lagrangian representation $\rho(t,x)=\rho_0(X(0,t,x))$ with the flow map $X$ solving the ODEs
\begin{equation}\label{eq-ODEs}
\left\{\begin{split}
\partial_tX(t,\tau,x)&=\boldsymbol{u}(t,X(t,\tau,x)),\\
X(\tau,\tau,x)&=x.
\end{split}\right.
\end{equation}
This link results in the existence and uniqueness of the PDE \eqref{eq-TE} by the Cauchy-Lipschitz (Picard-Linder\"of) theory for the ODE \eqref{eq-ODEs} (The Lipschitz condition on the vector fields can be replaced by one-side Lipschitz, $\log$-Lipschitz or Osgood condition, we refer to \cite{AL93} for the elaborate surveys). 

For non-Lipschitz vector fields, the link is less obvious and the uniqueness issue of \eqref{eq-TE} becomes subtler. The first breakthrough result traces back to the celebrated work of DiPerna and Lions \cite{DL89ODE}, they proved that for all $p\in[1,\infty]$, the Lagrangian representation and uniqueness hold in the class $L^{\infty}_tL^{p}_x$ for any given vector field $\boldsymbol{u}\in L^1_tW^{1,p'}_x$ with $\operatorname{div}\boldsymbol{u}\in L^{\infty}$, where $p'$ is the H\"older conjugate exponent of $p$. This result was extended to the case $\rho\in L^{\infty}_{t,x},\boldsymbol{u}\in L^1_tBV$ by Ambrosio \cite{Amb04} with deep tools from geometric measure theory. Whereafter, there are abundant researches for the theory of non-smooth vector fields and their applications on non-linear PDEs. For instance, the analogous results of \cite{DL89ODE,Amb04} have been established, for vector fields with gradient given by a singular integral in \cite{BC13} and for nearly incompressible $BV$ vector fields in \cite{BB19}, and the results in \cite{BC13} has been adapted to obtain the Lagrangian solutions for the Vlasov-Poisson system with $L^1$ data \cite{BBC16}, see also \cite{ACF17}. We refer to the works \cite{CL02,CLR03,LL04,CD08,CC16,CC21} and the surveys \cite{AC08,De08,Cri09,Amb17} for other important progresses and related results in this direction. 

For the non-uniqueness issue of \eqref{eq-TE}, two examples were given by DiPerna and Lions \cite{DL89ODE}, for an autonomous vector field $\boldsymbol{u}\in W^{1,p'}$ but $\operatorname{div}\boldsymbol{u}\notin L^{\infty}$ and for an autonomous divergence-free vector field $\boldsymbol{u}\in \cap_{0\le s<1}W^{s,1}$ but $\boldsymbol{u}\notin W^{1,1}$. Much later, based on the work \cite{Aiz78}, Depauw \cite{Dep03} constructed an example of a divergence-free vector field $\boldsymbol{u}\in L^1_{\rm loc}((0,T];BV(\mathbb{R}^2))$ but $\boldsymbol{u}\notin L^1(0,T;BV(\mathbb{R}^2))$, in the class of bounded densities. See also \cite{AF09,YZ17,ACM19}. In \cite{ABC13,ABC14},  Alberti, Bianchini and Crippa showed an example of an autonomous divergence-free vector field $\boldsymbol{u}\in \cap_{0\le \alpha<1}C^{0,\alpha}(\mathbb{R}^2)$ but $\boldsymbol{u}\notin C^{0,1}(\mathbb{R}^2)$, in the class of bounded densities. More recently, based on anomalous dissipation and mixing, Drivas, Elgindi, Iyer and Jeong proved non-uniqueness in the class $\rho\in L^{\infty}_tL^2_x$ for $\boldsymbol{u}\in L^1_tC^{1-}$ \cite{DEIJ20}. Roughly speaking, these results are proved by showing the non-uniqueness of the flow maps (the ODE level), with the violation of either the boundedness of $\operatorname{div}\boldsymbol{u}$ or the spacial $W^{1,1}_{\rm loc}$ regularity of $\boldsymbol{u}$.

Another approach is based on the convex integration technique (for complete summaries on this enormous theory, we refer to the surveys articles \cite{DS12,DS17b,DS19,BV19,BV21}), which gives counterexamples of uniqueness directly at the PDE level. The first non-uniqueness result with this technique was obtained by Crippa et al in \cite{CGSW15} using the framework of \cite{DS09}, but for vector field $\boldsymbol{u}$ was merely bounded. Later, a breakthrough result for the Sobolev vector field was obtained by Modena and Sz\'ekelyhidi \cite{MS18}. Based on their work, a series of works for the non-uniqueness to \eqref{eq-TE} have been done recently, we list the main functional setting of them below:  
\begin{enumerate}[(I).]
\item\cite{MS18}(Modena and Sz\'ekelyhidi): $\rho\in C_tL^p_x$, $\boldsymbol{u}\in C_t(W^{1,q}_x\cap L^{p'}_x)$ for $\frac{1}{p}+\frac{1}{q}>1+\frac{1}{d-1}$, $p>1$ and $d\ge 3$.
\item\cite{MS19}(Modena and Sz\'ekelyhidi): $\rho\in C_tL^1_x$, $\boldsymbol{u}\in C_tW^{1,q}_x\cap C_{t,x}$ for $q<d-1$ and $d\ge 3$.(The case $p=1$ of \cite{MS18})
\item\cite{MS20}(Modena and Sattig): $\rho\in C_tL^p_x$, $\boldsymbol{u}\in C_t(W^{1,q}_x\cap L^{p'}_x)$ for $\frac{1}{p}+\frac{1}{q}>1+\frac{1}{d}$ and $d\ge 2$. 
\item\cite{BCD21}(Bru\`e, Colombo and De Lellis): $\rho>0$, $\rho\in C_tL^p_x$, $\boldsymbol{u}\in C_t(W^{1,q}_x\cap L^{p'}_x)$ for $\frac{1}{p}+\frac{1}{q}>1+\frac{1}{d}$, $p>1$ and $d\ge 2$.
\item\cite{CL21}(Cheskidov and Luo): $\rho\in L^1_tL^p_x$, $\boldsymbol{u}\in L^1_tW^{1,q}_x\cap L^{\infty}_tL^{p'}_x$ for $\frac{1}{p}+\frac{1}{q}>1$, $p>1$ and $d\ge 3$.
\end{enumerate}
Comments:

Very recently, \cite{CC21} proves that the uniqueness result holds for $\boldsymbol{u}\in L^{\infty}_tW^{1,q}_x$ with $q>d$ in the class $\rho\in L^1_{t,x}$ under the assumption that the so called forward-backward integral curves of $\boldsymbol{u}$ are trivial. Whereas, the non-uniqueness result of \cite{MS20} holds for $\boldsymbol{u}\in L^{\infty}_tW^{1,q}_x$ with $q<d$ in the class $\rho\in C_tL^1_{x}$. Hence, roughly speaking when $p=1$, the regularity condition $\boldsymbol{u}\in L^{\infty}_tW^{1,q}_x$ with $q>d$ is optimal for the uniqueness of the weak solution in the class $L^1_{t,x}$. 

When $p>1$, the uniqueness is unclear in the range $1<\frac{1}{p}+\frac{1}{q}\le 1+\frac{1}{d}$ for $\rho\in C_tL^p_x$, $\boldsymbol{u}\in C_t(W^{1,q}_x\cap L^{p'}_x)$. However, \cite[Theorem 1.5]{BCD21} states that the uniqueness result holds for positive density $\rho\in L^{\infty}_tL^p_x$ and $\boldsymbol{u}\in L^1_tW^{1,q}_x$ with $\frac{1}{p}+\frac{1}{q}<1+\frac{q-1}{(d-1)q}$, which goes beyond the DiPerna-Lions range.

In \cite{CL21}, the non-uniqueness with the sharp spacial integrability ($1/p+1/q>1$) has been reached but at the expensive of time regularity. They proved the result using the convex integration scheme of \cite{MS18} combined with temporal intermittency and oscillations in the spirit of \cite{BV19b}. We refer to \cite{BV19b,CL21,BV21} for a more comprehensive discussion of the temporal intermittency.

A few months after we finished this paper, we learn that Cheskidov and Luo post a extremely nice result \cite{CL22} on arXiv, in which they show the non-uniqueness result for $u\in L_t^1W_x^{1,p}$, $\forall p<\infty$, $\rho\in\cap_{p<\infty,k\in\mathbb{N}}L_t^pC^k$, it also implies that the time-integrability assumption in the uniqueness of the DiPerna-Lions theory is sharp. In view of this surprising result, we have reconsidered the implications of our study, see Remark~\ref{rem-CLdiff}.

\subsection{Main results}

Based on the frameworks of \cite{MS20,CL21}, we obtain the following main result in this paper, which indicates that the non-uniqueness happens even in the range $1/p+1/q>1-\frac{p-1}{4(p+1)p}$, as long as $1\le s<\infty$, $p>1$.
\begin{theorem}\label{thm1.1}
Let $d\ge3$ and $p,q,s,\tilde{s}\in[1,\infty)$, $p',s'$ are H\"older conjugate exponents of $p,s$ respectively, satisfying $p>1$ and
\begin{equation}\label{eq-assum}
\frac{1}{p}+\frac{1}{q}>1-\frac{p-1}{4(p+1)p},\quad \tilde{s}\le \frac{s}{s/s'+(1+2\beta)/(1+4\beta)}
\end{equation}
with $\beta=(p-1)(d-1)/4p(p+1)$.

Then there exists a divergence-free vector field
\begin{equation*}
\boldsymbol{u}\in L^{\tilde{s}}(0,T;W^{1,q}(\mathbb{T}^d))\cap L^{s'}(0,T;L^{p^{\prime}}(\mathbb{T}^d)),
\end{equation*}
such that the uniqueness of \eqref{eq-TE} fails in the class $\rho\in L^s(0,T;L^{p}(\mathbb{T}^d))$.
\end{theorem}

\begin{remark}
On the one hand, Theorem~\ref{thm1.1} means that even for the spacial integrability higher than the DiPerna-Lions regime $(1-\frac{p-1}{4(p+1)p}<1/p+1/q<1)$, non-uniqueness still happens, which might be new in contrast to the previous works \cite{MS20,BCD21,CL21}. On the other hand, notice for $p>1$ and $1/s+1/s'=1$,
\begin{equation*}
s'>\frac{s}{s/s'+(1+2\beta)/(1+4\beta)}>1,
\end{equation*}
$\tilde{s}$ might be taken greater than 1, hence we improve the result of \cite{CL21} in two aspects: the restrictions on time integrability ($s,\tilde{s}$) and on spacial integrability ($p,q$).
\end{remark}

\begin{remark}
After suitable parameters setting, the condition \eqref{eq-assum} might be replaced by
\begin{equation}\label{eq-assum-2}
\frac{1}{p}+\frac{1}{q}>\frac{1}{p},\quad \tilde{s}\equiv 1,
\end{equation}
i.e., we allow the spatial integrability of $u$ can be arbitrary $q\in[1,\infty)$. However, since the low time integrability $\tilde{s}\equiv 1$, Cheskidov and Luo's work \cite{CL22} fully covers our result in this case. We mention that a comfortable condition on $\tilde{s}$ might be $\tilde{s}\le s'$. However, it can not be reached in this paper, since the second condition in  \eqref{eq-assum} or \eqref{eq-assum-2} plays a significant role in the proof of Lemma~\ref{lem-924-4.8} (When $\tilde{s}\to s'$, we need to set $\alpha\to 0$ in the proof of Lemma~\ref{lem-924-4.8}, then the balance of parameters setting in Sec.\ref{sec.4.3} will be upset).
\end{remark}

\begin{remark}\label{rem-CLdiff}
Cheskidov and Luo's work \cite{CL22} is very exciting, since they improve extremely the spatial regularity for the non-uniqueness. However, this is achieved by abandoning all the time regularity of the vector field $u$ except $L^1$ integrability. A significant difference between Theorem~\ref{thm1.1} and Cheskidov, Luo's result \cite{CL22} is that we can provide non-uniqueness with a little higher temporal integrability of $\nabla u$ ($\tilde{s}>1$). To the best of the authors' knowledge, there is no result of the non-uniqueness for $u\in L_t^{\tilde{s}}W_x^{1,q}\cap L_t^{s'}L_x^{1,p'}$, $\rho\in L_t^{s}L_x^p$ under the range $1/q+1/p\le 1+1/d$, $1<\tilde{s}\le s'<\infty$ before Theorem~\ref{thm1.1}. Notice also when $\tilde{s}=s'=\infty$, \cite{CC21} has proved the uniqueness for $\boldsymbol{u}\in L^{\infty}_tW^{1,q}_x$ with $q>d$, $\rho\in L^1_{t,x}$ under some additional assumptions of integral curves of $\boldsymbol{u}$. Hence it might be an interesting and valuable problem to consider uniqueness/non-uniqueness under high temporal integrability of $\nabla u$ and high spatial integrability of $\rho\nabla u$ ($1/p+1/q\le 1+1/d$).
\end{remark}

\begin{remark}
It seems possible to extend Theorem~\ref{thm1.1} to the border case $p=1$ and to the transport-diffusion equation
\begin{equation*}
\partial_t\rho+\boldsymbol{u}\cdot\nabla\rho-\Delta\rho=0
\end{equation*}
by utilizing the technique in \cite{MS19,MS20}. We finally mention that the two dimensional case $d=2$ is not treated in this paper due to the stationary Mikado flow, see Remark~\ref{rem.3.4}. However, thanks to the recent tricks given by Cheskidov and Luo \cite{CL22}, our result can be extended to the two dimensional case without difficulty. 
\end{remark}

Notice $1-\frac{p-1}{4(p+1)p}<1$ as long as $p>1$. Hence, combine the uniqueness result in \cite[Corollary II.1]{DL89ODE}, we obtain immediately from Theorem~\ref{thm1.1} that at least in the following sense, the $L^{\infty}$ in time of the density $\rho$ is critical for the uniqueness of weak solutions to \eqref{eq-TE}.

\begin{corollary}\label{coro1.2}
Let $d\ge3$, $p\in(1,\infty)$ and $s\in[1,\infty]$. The following holds.
\begin{enumerate}[(i).]
\item If $s<\infty$, then there exists a divergence-free vector field
\begin{equation*}
\boldsymbol{u}\in L^{1}(0,T;W^{1,p'}(\mathbb{T}^d))\cap L^{s'}(0,T;L^{p'}(\mathbb{T}^d)),
\end{equation*}
such that the uniqueness of \eqref{eq-TE} fails in the class $L^s(0,T;L^{p}(\mathbb{T}^d))$.
\item If $s=\infty$, then for any divergence-free field $\boldsymbol{u}\in L^{1}(0,T;W^{1,p'}(\mathbb{T}^d))$ and any initial data $\rho_0\in L^{p}(\mathbb{T}^d)$, there exists a uniqueness solution of \eqref{eq-TE} in the class $L^{\infty}(0,T;L^{p}(\mathbb{T}^d))$.
\end{enumerate}
\end{corollary}


We identify $[0,T]$ with an 1-dimensional torus and the time-periodic function $f$ on $[0,T]$ means $f(t+nT)=f(t)$ for all $n\in\mathbb{Z}$. Theorem~\ref{thm1.1} follows immediately from the following theorem.
\begin{theorem}\label{thm1.3}
Let $d\ge3$ and $p,q,s,\tilde{s}\in[1,\infty)$, $p',s'$ are H\"older conjugate exponents of $p,s$ respectively, satisfying $p>1$ and \eqref{eq-assum}. For any $\epsilon>0$ and  any time-periodic $\tilde\rho\in C^{\infty}([0,T]\times\mathbb{T}^d)$ with constant mean
\begin{equation*}
\fint_{\mathbb{T}^d}\tilde\rho(t,x)\,dx=\fint_{\mathbb{T}^d}\tilde\rho(0,x)\,dx\text{ for all }t\in[0,T],
\end{equation*}
there exists a divergence-free vector field $\boldsymbol{u}$ and a density $\rho$ such that the following holds.
\begin{enumerate}[(i).]
\item $\boldsymbol{u}\in L^{\tilde{s}}(0,T;W^{1,q}(\mathbb{T}^d))\cap L^{s'}(0,T;L^{p^{\prime}}(\mathbb{T}^d))$ and $\rho\in L^s(0,T;L^{p}(\mathbb{T}^d))$.
\item $\rho(t)$ is continuous in the distributional sense and for $t=0,T$, $\rho(t)=\tilde\rho(t)$.
\item $(\rho,\boldsymbol{u})$ is a weak solution to \eqref{eq-TE} with initial data $\tilde\rho(0)$.
\item The deviation of $L^p$ norm is small on average: $\|\rho-\tilde\rho\|_{L^s_tL^p_x}\le\epsilon$.
\end{enumerate}
\end{theorem}

\noindent\textbf{Proof of Theorem~\ref{thm1.1}.} Let $\bar{\rho}\in C_0^{\infty}(\mathbb{T}^d)$ with $\|\bar{\rho}\|_{L^p_x}=1$. We take $\tilde\rho=\chi(t)\bar{\rho}(x)$ with $\chi\in C^{\infty}([0,T];[0,1])$ satisfying $\chi=1$ if $\vert  t-\frac{T}{2}\vert \le \frac{T}{4}$ and $\chi=0$ if $\vert  t-\frac{T}{2}\vert  \ge\frac{3T}{8}$. We apply Theorem~\ref{thm1.3} with $\epsilon=\frac{1}{4}(\frac{T}{4})^{1/s}$ and obtain $(\rho,\boldsymbol{u})$ solving \eqref{eq-TE} with $\rho_0\equiv0$. By the choice of $\epsilon$, we claim that $\rho$ cannot have a constant $L^p_x$ norm and obviously $\rho\not\equiv 0$, which implies the non-uniqueness (as well as the existence of non-renormalized solution).

Indeed, assume $\|\rho(t)\|_{L^p}\equiv C$ for some $C>0$. On the one hand, due to $\|\rho-\tilde\rho\|_{L^s_tL^p_x}\le \frac{1}{4}(\frac{T}{4})^{1/s}$, we have
\begin{equation*}\frac{T}{2}\vert C-1\vert \le\int_{\frac{T}{4}}^{\frac{3T}{4}}\|\rho(t)-\tilde\rho(t)\|_{L^p_x}\,dt\le\big(\frac{T}{2}\big)^{1/s'}\|\rho-\tilde\rho\|_{L^s_tL^p_x}\le\frac{T}{8}, \end{equation*}
hence $\vert C-1\vert \le 1/4$, which implies $C>3/4$.

On the other hand,
\begin{equation*}
\frac{T}{8}\ge\big(\frac{T}{4}\big)^{1/s'}\|\rho-\tilde\rho\|_{L^s_tL^p_x}\ge\int_{[0,\frac{T}{8}]\cup[\frac{7T}{8},T]}\|\rho(t)-\tilde\rho(t)\|_{L^p_x}\,dt=\frac{T}{4}C,
\end{equation*}
hence $C\le 1/2$, in contradiction with $C>3/4$, we prove the claim. \qed

\subsection{Notations}

The norms of $L^p(\mathbb{T}^d)$, $L^s(0,T)$, $L^s(0,T;L^p(\mathbb{T}^d))$ will be denoted standardly as $\|\cdot\|_{L^p_x}$, $\|\cdot\|_{L^s_t}$, $\|\cdot\|_{L^s_tL^p_x}$ or just $\|\cdot\|_{L^p}$ when there is no confusion. The norm of $C^k([0,T]\times\mathbb{T}^d)$ will be denoted as $\|\cdot\|_{C^k}$.

For any $f\in L^1(\mathbb{T}^d)$, its spacial mean is $\fint_{\mathbb{T}^d}f\,dx=\int_{\mathbb{T}^d}f\,dx$ and denote simply as $\fint f$. We denote $C_0^{\infty}(\mathbb{T}^d)$ as the space of smooth periodic functions with zero mean.

Denote the standard partial differential operators with multiindex ${\boldsymbol{k}}\in\mathbb{N}^d$ as $\partial^{\boldsymbol{k}}:=\partial_{x_1}^{k_1}\cdots\partial_{x_d}^{k_d}$. For any $f\in C^{\infty}(\mathbb{T}^d)$, the obvious facts are $\|\partial^{\boldsymbol{k}}(f(\sigma\cdot))\|_{L^p}=\sigma^{\vert \boldsymbol{k}\vert }\|\partial^{\boldsymbol{k}}f\|_{L^p}$ for any $\forall p\in[1,\infty)$, $\forall\sigma\in\mathbb{N}_+$ and of course $\partial^{\boldsymbol{k}}f\in C_0^{\infty}(\mathbb{T}^d)$ as long as ${\boldsymbol{k}}\ne 0$.

$a\lesssim b$ will be denoted as $a\le Cb$ with some inessential constant $C$. If the constant depends on some quantities, for instance $r$, it will be denoted as $C_{r}$, and $a\lesssim_r b$ means $a\le C_{r} b$.

$C_*$ represents a positive constant that might depend on the old solution $(\rho,\boldsymbol{u},\boldsymbol{R})$ and the constant $N$ but never on $\nu,\delta$ given in Proposition~\ref{prn3.1}. $C_*$ may change from line to line.

\section{Main Proposition and the Proof of Theorem~\ref{thm1.3}}\label{sec3}

Without loss of generality, in the rest of the paper, we assume $T=1$ and identify the time interval $[0,1]$ with an 1-dimensional torus.

We follow the framework of \cite{CL21} (see \cite{MS18} for the earliest version) to obtain space-time periodic approximate solutions $(\rho,\boldsymbol{u},\boldsymbol{R})$ to the transport equation by solving the continuity-defect equation
\begin{equation}\label{eq-CDE}
\left\{
\begin{split}
&\partial_t\rho+\operatorname{div}(\rho\boldsymbol{u})=\operatorname{div}\boldsymbol{R},\\
&\operatorname{div} \boldsymbol{u}=0,
\end{split}\right.
\end{equation}
where $\boldsymbol{R}\colon[0,1]\times\mathbb{T}^d\to\mathbb{R}^d$ is called the defect field. 

For any $0<r<1$, denote $I_r:=[r,1-r]$. To build a iteration scheme for proving Theorem~\ref{thm1.3}, we will construct the small perturbations on $I_r\times\mathbb{T}^d$ of $(\rho,\boldsymbol{u})$ to obtain a new solution $(\rho^1,\boldsymbol{u}^1,\boldsymbol{R}^1)$ such that the new defect field $\boldsymbol{R}^1$ has small $L^1_{t,x}$ norm. This is the following main proposition of the paper.

\begin{proposition}\label{prn3.1}
Let $d\ge3$ and $p,q,s,\tilde{s}\in[1,\infty)$, $p',s'$ are H\"older conjugates of $p,s$ respectively, satisfying $p>1$ and \eqref{eq-assum}. There exist a universal constant $M>0$ and a large integer $N\in\mathbb{N}$ such that the following holds.

Suppose $(\rho,\boldsymbol{u},\boldsymbol{R})$ is a smooth solution of \eqref{eq-CDE} on $[0,1]$. Then for any $\delta,\nu>0$, there exists another smooth solution $(\rho^1,\boldsymbol{u}^1,\boldsymbol{R}^1)$ of \eqref{eq-CDE} on $[0,1]$ which fulfills the estimates
\begin{align*}
\|\rho^1-\rho\|_{L^s_tL^p_x}&\le\nu M\|\boldsymbol{R}\|^{1/p}_{L^1_{t,x}},\\
\|\boldsymbol{u}^1-\boldsymbol{u}\|_{L^{s'}_tL^{p'}_x}&\le\nu^{-1} M\|\boldsymbol{R}\|^{1/p'}_{L^1_{t,x}},\\
\|\boldsymbol{u}^1-\boldsymbol{u}\|_{L^{\tilde s}_tW^{1,q}_x}&\le\delta,\quad\|\boldsymbol{R}^1\|_{L^1_{t,x}}\le\delta.
\end{align*}
In addition, the density perturbation $\rho^1-\rho$ has zero spacial mean and satisfies
\begin{align}
\Big\vert  \int_{\mathbb{T}^d}(\rho^1-\rho)(t,x)\phi(x)\,dx\Big\vert  &\le\delta\|\phi\|_{C^N}\quad\forall t\in[0,1],\,\forall\phi\in C^{\infty}(\mathbb{T}^d),\label{eq-10}\\
\operatorname{supp}_t(\rho^1-\rho)&\in I_r\,\text{ for some }r>0.\label{eq-11}
\end{align}
\end{proposition}

\noindent\textbf{Proof of Theorem~\ref{thm1.3}.} Assume $T=1$. We will construct a sequence $(\rho^n,\boldsymbol{u}^n,\boldsymbol{R}^n)$ of solutions to \eqref{eq-CDE}. For $n=1$, we set
\begin{equation*}(\rho^1,\boldsymbol{u}^1,\boldsymbol{R}^1):=(\tilde\rho,0,\mathcal{R}(\partial_t\tilde\rho)).\end{equation*}
Notice the constant mean assumption on $\tilde\rho$ implies zero mean of $\partial_t\tilde\rho$, hence $(\rho^1,\boldsymbol{u}^1,\boldsymbol{R}^1)$ solves \eqref{eq-CDE}.

Next we apply Proposition~\ref{prn3.1} inductively to obtain $(\rho^n,\boldsymbol{u}^n,\boldsymbol{R}^n)$ for $n=2,3\cdots$ as follows. Set $\nu_1:=\epsilon(2M\|\boldsymbol{R}^1\|^{1/p}_{L^1_{t,x}})^{-1}$ and choose sequence $\{(\delta_n,\nu_n)\}_{n=2}^{\infty}\subset(0,\infty)^2$ such that $\sum_{n}\delta_n^{1/2}=1$, $\delta_n^{1/p}\nu_n=\epsilon\delta^{1/2}_n/2M$. Observe that $\delta_n^{1/p'}/\nu_n=2M\delta^{1/2}_n/\epsilon$.

Given $(\rho^n,\boldsymbol{u}^n,\boldsymbol{R}^n)$, we apply Proposition~\ref{prn3.1} with parameters $\nu=\nu_n$ and $\delta=\delta_{n+1}$ to obtain a new triple $(\rho^{n+1},\boldsymbol{u}^{n+1},\boldsymbol{R}^{n+1})$ which verifies
\begin{align*}
\|\rho^{n+1}-\rho^n\|_{L^s_tL^p_x}&\le\nu_{n} M\|\boldsymbol{R}^n\|^{1/p}_{L^1_{t,x}},\\
\|\boldsymbol{u}^{n+1}-\boldsymbol{u}^n\|_{L^{s'}_tL^{p'}_x}&\le\nu_{n}^{-1} M\|\boldsymbol{R}^n\|^{1/p'}_{L^1_{t,x}},\\
\|\boldsymbol{u}^{n+1}-\boldsymbol{u}^n\|_{L^{\tilde s}_tW^{1,q}_x}&\le\delta_{n+1},\quad\|\boldsymbol{R}^{n+1}\|_{L^1_{t,x}}\le\delta_{n+1},\\
\Big\vert \int_{\mathbb{T}^d}(\rho^{n+1}-\rho^n)(t,x)\phi(x)\,dx\Big\vert &\le\delta_{n+1}\|\phi\|_{C^N}\quad\forall t\in[0,1],\,\forall\phi\in C^{\infty}(\mathbb{T}^d).
\end{align*}
When $n\ge 2$, we have
\begin{align*}
\|\rho^{n+1}-\rho^n\|_{L^s_tL^p_x}&\le\frac{\epsilon\delta^{1/2}_n}{2},\\
\|\boldsymbol{u}^{n+1}-\boldsymbol{u}^n\|_{L^{s'}_tL^{p'}_x}&\le\frac{2M^2\delta^{1/2}_n}{\epsilon}.
\end{align*}
Clearly there are functions $\rho\in L^s_tL^p_x$ and $\boldsymbol{u}\in L^{s'}_tL^{p'}_x\cap L^{\tilde s}_tW^{1,q}_x$ such that $\rho^{n}\to\rho$ in $L^s_tL^p_x$ and $\boldsymbol{u}^n\to\boldsymbol{u}$ in $L^{s'}_tL^{p'}_x\cap L^{\tilde s}_tW^{1,q}_x$. Moreover, we have $\rho^{n}\boldsymbol{u}^n\to\rho\boldsymbol{u}$ and $\boldsymbol{R}^n\to 0$ in $L^1_{t,x}$, and  $\int_{\mathbb{T}^d}\rho^n(\cdot,x)\phi(x)\,dx\to\int_{\mathbb{T}^d}\rho(\cdot,x)\phi(x)\,dx$ in $L^{\infty}_t$. Combine the fact $\operatorname{supp}_t(\rho^{n+1}-\rho^n)\in I_{r_n}$ for some $r_n>0$, we obtain the temporal continuity of $\rho$ in the distributional sense and for $t=0,1$, $\rho(t)=\tilde\rho(t)$, furthermore $(\rho,\boldsymbol{u})$ is a weak solution to \eqref{eq-TE} with initial data $\tilde\rho(0)$.

Finally, thanks to the choice of $\{\delta_n\},\{\nu_n\}$, we have
\begin{equation*}
\|\rho-\tilde\rho\|_{L^s_tL^p_x}\le\|\rho^{2}-\rho^1\|_{L^s_tL^p_x}+\sum_{n=2}^{\infty}\|\rho^{n+1}-\rho^n\|_{L^s_tL^p_x}\le\frac{\epsilon}{2}+\frac{\epsilon}{2}\sum_{n=2}^{\infty}\delta^{1/2}_n=\epsilon.
\end{equation*}\qed

%
%


\section{Technical tools}\label{sec2}

In this section, we collect the technical tools prepared in \cite{MS18,CL21} for the proof of the main proposition (Proposition~\ref{prn3.1}). We refer to \cite{MS18,MS20,CL21} for more details.

\subsection{Anti-divergence operators}

By the classical Fourier analysis, for any $f\in C^{\infty}(\mathbb{T}^d)$, a unique solution in $C_0^{\infty}(\mathbb{T}^d)$ of the Poisson equation
\begin{equation*}
\Delta u=f-\fint f
\end{equation*}
is given by $u(x)=\sum_{k\in\mathbb{Z}^d\setminus\{0\}}(4\pi\vert k\vert ^2)^{-1}e^{2\pi ik\cdot x}\hat{f}(k)$. Hence the standard anti-divergence operator $\mathcal{R}\colon C^{\infty}(\mathbb{T}^d)\to C^{\infty}_0(\mathbb{T}^d;\mathbb{R}^d)$ can be defined as
\begin{equation*}
\mathcal{R}f:=\Delta^{-1}\nabla f,
\end{equation*}
which satisfies
\begin{equation*}
\operatorname{div}(\mathcal{R}f)=f-\fint f.
\end{equation*}
Obviously, for every $\sigma\in\mathbb{N}_+$ and $f\in C_0^{\infty}(\mathbb{T}^d)$ there holds
\begin{equation*}
\mathcal{R}[f(\sigma\cdot)](x)=\sigma^{-1}(\mathcal{R}f)(\sigma x).
\end{equation*}

Further, the first order bilinear anti-divergence operator $\mathcal{B}\colon C^{\infty}(\mathbb{T}^d)\times C^{\infty}(\mathbb{T}^d)\to C^{\infty}(\mathbb{T}^d;\mathbb{R}^d)$ can be defined as
\begin{equation*}
\mathcal{B}(a,f):=a\mathcal{R}f-\mathcal{R}(\nabla a\cdot\mathcal{R}f),
\end{equation*}
which satisfies
\begin{equation*}
\operatorname{div}(\mathcal{B}(a,f))=af-\fint af\text{ provided that }f\in C^{\infty}_0(\mathbb{T}^d).
\end{equation*}
The bilinear anti-divergence operator $\mathcal{B}$ has the additional advantage of gaining derivative from $f$ when $f$ has zero mean and a very small period. See also higher order variants in \cite{MS20}.

\begin{remark}
Notice the definitions of $\mathcal{R},\mathcal{B}$ in \cite{MS20} are slightly different from the definitions in this paper, which actually are defined as
\begin{equation*}
\mathcal{R}f=\nabla\Delta^{-1} f,\quad \mathcal{B}(a,f)=a\mathcal{R}f-\mathcal{R}(\nabla a\cdot\mathcal{R}f-\fint af).
\end{equation*}
In this case, $f$ must be mean zero.
\end{remark}

\begin{lemma}[{\cite[Lemma 2.1]{CL21}}]\label{lem2.1}
Let $d\ge2$. For every $m\in\mathbb{N}$ and $r\in[1,\infty]$, the anti-divergence operator $\mathcal{R}$ is bounded on $W^{m,r}(\mathbb{T}^d)$:
\begin{equation}
\|\mathcal{R}f\|_{W^{m,r}}\lesssim\|f\|_{W^{m,r}}.
\end{equation}
Moreover for all $f\in C^{\infty}(\mathbb{T}^d)$ and $1<r<\infty$, the Calder\'on-Zygmund inequality holds:
\begin{equation}\label{eq-5}
\|\mathcal{R}(\operatorname{div} f)\|_{L^r}\lesssim\|f\|_{L^r}.
\end{equation}
\end{lemma}

\begin{lemma}[{\cite[Lemma 2.2]{CL21}}]\label{lem2.2}
Let $d\ge2$ and $r\in[1,\infty]$. Then for any $a,f\in C^{\infty}(\mathbb{T}^d)$:
\begin{equation*}
\|\mathcal{B}(a,f)\|_{L^r}\lesssim\|a\|_{C^1}\|\mathcal{R}f\|_{L^r}.
\end{equation*}
\end{lemma}

\begin{lemma}[{\cite[Lemma 2.4]{CL21}}]\label{lem2.3}
Let $\sigma\in\mathbb{N}$ and $a,f\in C^{\infty}(\mathbb{T}^d)$. Then for all $r\in[1,\infty]$,
\begin{equation*}
\Big\vert \|a(\cdot)f(\sigma\cdot)\|_{L^r}-\|a\|_{L^r}\|f\|_{L^r}\Big\vert \le C_{r} \sigma^{-1/r}\|a\|_{C^1}\|f\|_{L^r}.
\end{equation*}
\end{lemma}
\begin{remark}
Lemma~\ref{lem2.3} is called the improved H\"older inequality, which is established in \cite[Lemma 2.1]{MS18}, inspired by \cite[Lemma 3.7]{BV19b}.
\end{remark}

\begin{lemma}[{\cite[Lemma 2.5]{CL21}}]\label{lem2.4}
Let $\sigma\in\mathbb{N}$, $a\in C^{\infty}(\mathbb{T}^d)$ and $f\in C_0^{\infty}(\mathbb{T}^d)$. Then for all even $n\ge0$
\begin{equation*}
\Big\vert  \fint_{\mathbb{T}^d}a(x)f(\sigma x)\,dx\Big\vert  \lesssim_n \sigma^{-n}\|a\|_{C^n}\|f\|_{L^2}.
\end{equation*}
\end{lemma}
\begin{remark}
Lemma~\ref{lem2.4} is used to estimate the space mean of the perturbations and obtain that the density perturbations converges to 0 in the distributional sense in space, uniformly in time, see \eqref{eq-10}.
\end{remark}

\subsection{Mikado flows}

Now we define the Mikado flow $(\Phi_j^{\mu},\boldsymbol{W}_j^{\mu})$ given in \cite{CL21}, which are a family of periodic stationary solutions to the transport equation \eqref{eq-TE}. Where $\Phi_j^{\mu}$ is the Mikado density, $\boldsymbol{W}_j^{\mu}$ is the Mikado filed.

Let $d\ge 3$. Fix $\boldsymbol{\Omega}\in C_c^{\infty}(\mathbb{R}^{d-1};\mathbb{R}^{d-1})$ satisfying
\begin{equation*}
\operatorname{supp}\boldsymbol{\Omega}\subset(0,1)^{d-1},\quad\int_{\mathbb{R}^{d-1}}(\operatorname{div}\boldsymbol{\Omega})^2\,dx=1.
\end{equation*}
Denote $\phi=\operatorname{div}\boldsymbol{\Omega}$ and $\boldsymbol{\Omega}^{\mu}(x)=\boldsymbol{\Omega}(\mu x)$, $\phi^{\mu}(x)=\phi(\mu x)$. We can define a family of non-periodic stationary solutions to \eqref{eq-TE} as
\begin{align*}
\tilde{\Phi}_j^{\mu}(x)&=\mu^{\frac{d-1}{p}}\phi^{\mu}(x_1,\dots,x_{j-1},x_{j+1},\dots,x_d),\\
\tilde{\boldsymbol{W}}_j^{\mu}(x)&=\mu^{\frac{d-1}{p'}}\phi^{\mu}(x_1,\dots,x_{j-1},x_{j+1},\dots,x_d)\boldsymbol{e}_j.
\end{align*}
The potential is defined as
\begin{equation*}
\tilde{\boldsymbol{\Omega}}_j^{\mu}(x)=\mu^{-1+\frac{d-1}{p}}\boldsymbol{\Omega}^{\mu}(x_1,\dots,x_{j-1},x_{j+1},\dots,x_d).
\end{equation*}


The periodic solutions $(\Phi_j^{\mu},\boldsymbol{W}_j^{\mu})$ with mutually disjoint supports can be constructed by translation and periodization of the non-periodic flow $(\tilde{\Phi}_j^{\mu},\tilde{\boldsymbol{W}}_j^{\mu})$. This is based on the geometrical fact that along any two directions in $\mathbb{R}^d$ for $d\ge 3$, there exist two disjoint lines.

For instance, notice $\{\operatorname{supp}\tilde{\Phi}_j^{\mu}\}_{j=1}^d$ are $d$ cylinders with side length $1/\mu$ lying at $x_j$-axis respectively. When $\mu\ge 8d$, one can move the $j$-th cylinder with length $\frac{1}{4d}$ along a direction $\boldsymbol{p}_j\in\{\boldsymbol{e}_j\}_{j=1}^d$, such that all cylinders are mutually disjoint and keep lying in $\cup_{j=1}^d[(0,1)^d+\mathbb{R}\boldsymbol{e}_j]$. By means of the Possion summation, we define periodic Mikado flow as
\begin{align*}
\Phi_j^{\mu}(x)&=\sum_{\boldsymbol{n}\in\mathbb{Z}^d,\boldsymbol{n}_j=0}\tilde{\Phi}_j^{\mu}\Big(x-\frac{1}{4d}\boldsymbol{p}_j+\boldsymbol{n}\Big),\\
\boldsymbol{W}_j^{\mu}(x)&=\sum_{\boldsymbol{n}\in\mathbb{Z}^d,\boldsymbol{n}_j=0}\tilde{\boldsymbol{W}}_j^{\mu}\Big(x-\frac{1}{4d}\boldsymbol{p}_j+\boldsymbol{n}\Big),
\end{align*}
and the periodic potential is defined as
\begin{equation*}
\boldsymbol{\Omega}_j^{\mu}(x)=\sum_{\boldsymbol{n}\in\mathbb{Z}^d,\boldsymbol{n}_j=0}\tilde{\boldsymbol{\Omega}}_j^{\mu}\Big(x-\frac{1}{4d}\boldsymbol{p}_j+\boldsymbol{n}\Big).
\end{equation*}

For the Mikado flow, we have the following proposition.

\begin{proposition}[{\cite[Proposition 4.3, Theorem 4.4]{CL21}}]\label{prn2.5}
Let $d\ge3$ and $\mu\ge 8d$. Then the periodic functions $\Phi_j^{\mu},\boldsymbol{W}_j^{\mu}\in C^{\infty}_0(\mathbb{T}^d)$, $\boldsymbol{\Omega}_j^{\mu}\in C^{\infty}(\mathbb{T}^d)$ verify the following.
\begin{enumerate}[(i).]
\item For any $1\le r\le\infty$, $m\in\mathbb{N}$,
\begin{equation}\label{eq-6}
\begin{split}
\|\nabla^m\Phi_j^{\mu}\|_{L^r}&\lesssim_m\mu^{m+\frac{d-1}{p}-\frac{d-1}{r}},\\
\|\nabla^m\boldsymbol{\Omega}_j^{\mu}\|_{L^r}&\lesssim_m\mu^{m-1+\frac{d-1}{p}-\frac{d-1}{r}},\\
\|\nabla^m\boldsymbol{W}_j^{\mu}\|_{L^r}&\lesssim_m\mu^{m+\frac{d-1}{p^{\prime}}-\frac{d-1}{r}}.
\end{split}
\end{equation}
\item $\Phi_j^{\mu},\boldsymbol{W}_j^{\mu},\boldsymbol{\Omega}_j^{\mu}$ solve
\begin{equation}\label{eq-7}
\left\{
\begin{split}
&\operatorname{div}(\Phi_j^{\mu}\boldsymbol{W}_j^{\mu})=0,\\
&\operatorname{div} \boldsymbol{W}_j^{\mu}=0,\quad\operatorname{div}\boldsymbol{\Omega}_j^{\mu}=\Phi_j^{\mu}.
\end{split}\right.
\end{equation}
\item There hold
\begin{equation}\label{eq-8}
\int_{\mathbb{T}^d}\Phi_j^{\mu}\boldsymbol{W}_j^{\mu}=\boldsymbol{e}_j\text{ for all } 1\le j\le d;\quad\Phi_j^{\mu}\boldsymbol{W}_{k}^{\mu}=0\text{ if }j\ne k,
\end{equation}
\end{enumerate}
where $\boldsymbol{e}_j$ is the $j$-th standard Euclidean basis.
\end{proposition}
\begin{remark}\label{rem.3.4}
The concept of Mikado flow was introduced in \cite{DS17} firstly and then adapted in \cite{MS18,MS19,MS20,BCD21,CL21} for the non-uniqueness results of the transport equation. Notice the construction of stationary Mikado flow requires the dimension $d$ is not less than three. An alternative is the space-time Mikado flow introduced in \cite{MS20} (see also \cite{BCD21}), which can be constructed in $\mathbb{T}^2$. However, the space-time flow brings new difficulty when applying the temporal intermittency.
\end{remark}

\subsection{Intermittent functions in time}

In this subsection, we define the intermittent oscillatory functions $\bar{g}_{\kappa}$ and $\tilde{g}_{\kappa}$. Take $g\in C_c^{\infty}(\mathbb{R})$ satisfying $\operatorname{supp} g\subset(0,1)$ and
\begin{equation*}
\int_{[0,1]}g^2\,dt=1.
\end{equation*}
For $\kappa\ge1$, define
\begin{equation*}
g_{\kappa}(t):=\sum_{n\in\mathbb{Z}}g(\kappa t+\kappa n),\quad \bar{g}_{\kappa}(t):=\kappa^{1/s^{\prime}}g_{\kappa}(t),\quad\tilde{g}_{\kappa}(t):=\kappa^{1/s}g_{\kappa}(t).
\end{equation*}
In this case, we have that $\bar{g}_{\kappa},\tilde{g}_{\kappa}\in C^{\infty}([0,1])$ be 1-periodic functions and the following facts (similar to Proposition~\ref{prn2.5}) hold
\begin{equation}\label{eq-12}
\int_{[0,1]}\bar{g}_{\kappa}\tilde{g}_{\kappa}\,dt=1,\quad\|\partial_t^m\bar{g}_{\kappa}\|_{L^r_t}\lesssim \kappa^{m+\frac{1}{s^{\prime}}-\frac{1}{r}},\quad\|\partial_t^m\tilde{g}_{\kappa}\|_{L^r_t}\lesssim \kappa^{m+\frac{1}{s}-\frac{1}{r}}.
\end{equation}

Define $h_{\kappa}(t):=\int_0^t(\bar{g}_{\kappa}\tilde{g}_{\kappa}-1)\,d\tau$, which obviously satisfies
\begin{equation}\label{eq-13}
\partial_th_{\kappa}=\bar{g}_{\kappa}\tilde{g}_{\kappa}-1,\quad\|h_{\kappa}\|_{L^{\infty}_t}\le 1.
\end{equation}

We write $\boldsymbol{R}=\sum_j R_j\boldsymbol{e}_j$, where $\boldsymbol{e}_j$ is the $j$-th standard Euclidean basis.

Recall the notation $I_r=[r,1-r]\subset(0,1)$ for $0<r<1$. Define the smooth cutoff functions $\chi_{j}\in C_c^{\infty}(\mathbb{R}\times\mathbb{T}^d)$ satisfying
\begin{equation}\label{eq-14}
0\le\chi_{j}\le1,\quad\chi_{j}(t,x)=\left\{
\begin{split}
0,&\text{ if }\vert R_j\vert \le \frac{\delta}{8d}\text{ or }t\notin I_{r/2},\\
1,&\text{ if }\vert R_j\vert \ge \frac{\delta}{4d}\text{ and }t\in I_{r}.
\end{split}\right.
\end{equation}
Where $r>0$ is fixed sufficiently small enough such that 
\begin{equation}\label{eq-15}
\|\boldsymbol{R}\|_{L^{\infty}_{t,x}}\le\frac{\delta}{8rd}.
\end{equation}
Notice $\operatorname{supp}\chi_{j}\subset I_{r/2}\subset(0,1)$. By a slight abuse of notation, $\chi_{j}$ denote the 1-periodic extension in time of $\chi_{j}$. Define $\widetilde{R}_j:=\chi_{j}R_j$.

\section{Proof of Proposition~\ref{prn3.1}}

In this section, we follow the lines in \cite{CL21} to construct the perturbations and defect field, and then finish the proof of Proposition~\ref{prn3.1}. In particular, we set that the concentration in time is stronger than in space, the oscillation in time is weaker than in space, see Sec.~\ref{sec.4.3}.

\subsection{Constructing perturbations}

We first define the principle part of the perturbations by the Mikado flows given in Proposition~\ref{prn2.5}. Let
\begin{equation}\label{eq-16}
\begin{split}
\theta_p(t,x)&:=\tilde{g}_{\kappa}(\lambda t)\sum\nolimits_{j}a_j(t,x)\Phi_j^{\mu}(\sigma x),\\
\boldsymbol{w}_p(t,x)&:=\bar{g}_{\kappa}(\lambda t)\sum\nolimits_{j}b_j(t,x)\boldsymbol{W}_j^{\mu}(\sigma x),
\end{split}
\end{equation}
where
\begin{align*}
&a_j(t,x):=\nu\Big(\frac{\|\widetilde{R}_j(t)\|_{L^1}}{\|\widetilde{R}_j\|_{L^{1}_{t,x}}}\Big)^{\frac{1}{s}-\frac{1}{p}}\operatorname{sign}(-R_j)\chi_j\vert R_j\vert ^{\frac{1}{p}},\\
&b_j(t,x):=\nu^{-1}\Big(\frac{\|\widetilde{R}_j(t)\|_{L^1}}{\|\widetilde{R}_j\|_{L^{1}_{t,x}}}\Big)^{\frac{1}{p}-\frac{1}{s}}\chi_j\vert R_j\vert ^{\frac{1}{p'}}.
\end{align*}
Notice $a_jb_j=-\chi_j^2R_j$. By \cite[Lemma 7.1]{CL21}, we have 
\begin{equation*}
a_j,b_j\in C^{\infty}([0,1]\times\mathbb{T}^d),\quad \|\widetilde{R}_j(t)\|_{L^1}\in C^{\infty}([0,1]),
\end{equation*}
and the following estimates hold true:
\begin{equation}\label{eq-17}
\begin{split}
&\|a_j(t)\|_{L^p}\le\nu\|\widetilde{R}_j\|_{L^{1}_{t,x}}^{\frac{1}{p}-\frac{1}{s}}\|\widetilde{R}_j(t)\|_{L^1}^{\frac{1}{s}},\\
&\|b_j(t)\|_{L^{p'}}\le\nu^{-1}\|\widetilde{R}_j\|_{L^{1}_{t,x}}^{\frac{1}{s}-\frac{1}{p}}\|\widetilde{R}_j(t)\|_{L^1}^{\frac{1}{s'}}.
\end{split}
\end{equation}
Moreover for any $k\in\mathbb{N}$, there exists a constant $C_{*}$ such that
\begin{equation}\label{eq-18}
\|a_j\|_{C^k}\le C_*\nu,\quad\|b_j\|_{C^k}\le C_*\nu^{-1}.
\end{equation}

Notice $\theta_p$ is not mean zero and $\boldsymbol{w}_p$ is not divergence-free. To make sure the zero mean of $\rho^1$ and divergence-free $\boldsymbol{u}^1$, we need the corrections of the perturbations. The correctors are defined by
\begin{align*}
\theta_c(t)&:=-\fint_{\mathbb{T}^d}\theta_p(t,x)\,dx,\\
\boldsymbol{w}_c(t,x)&:=-\bar{g}_{\kappa}(\lambda t)\sum\nolimits_{j}\mathcal{B}(\partial_jb_j,\boldsymbol{W}_j^{\mu}(\sigma x)\cdot\boldsymbol{e}_j).
\end{align*}
Notice by Proposition~\ref{prn2.5}, $\boldsymbol{W}_j^{\mu}$ has zero mean, we have
\begin{align*}
\operatorname{div}\mathcal{B}(\partial_jb_j,\boldsymbol{W}_j^{\mu}(\sigma x)\cdot\boldsymbol{e}_j)&=\partial_jb_j\boldsymbol{W}_j^{\mu}(\sigma x)\cdot\boldsymbol{e}_j-\fint \partial_jb_j\boldsymbol{W}_j^{\mu}(\sigma x)\cdot\boldsymbol{e}_j\\
&=\operatorname{div}(b_j\boldsymbol{W}_j^{\mu}(\sigma x)),
\end{align*}
hence
\begin{equation*}
\operatorname{div}\boldsymbol{w}_c(t,x)=-\bar{g}_{\kappa}(\lambda t)\sum\nolimits_{j}\operatorname{div}(b_j\boldsymbol{W}_j^{\mu}(\sigma x))=-\operatorname{div}\boldsymbol{w}_p(t,x).
\end{equation*}

Finally, to take advantage of the temporal oscillations, we define the temporal oscillator
\begin{equation*}
\theta_o(t,x):=\lambda^{-1} h_{\kappa}(\lambda t)\operatorname{div}\sum\nolimits_{j}\chi_j^2R_j\boldsymbol{e}_j.
\end{equation*}
Notice by definition, $\theta_o$ has zero mean.

Now we are able to define the perturbations by
\begin{equation}
\theta:=\theta_p+\theta_c+\theta_o,\quad\boldsymbol{w}:=\boldsymbol{w}_p+\boldsymbol{w}_c,
\end{equation}
and $\rho^1,\boldsymbol{u}^1$ are defined by
\begin{equation}
\rho^1:=\rho+\theta,\quad \boldsymbol{u}^1:=\boldsymbol{u}+\boldsymbol{w}.
\end{equation}

\subsection{Constructing the defect field}

Now we define the new defect field $\boldsymbol{R}^1$ satisfying the continuity-defect equation
\begin{equation*}
\partial_t\rho^1+\boldsymbol{u}^1\cdot\nabla\rho^1=\operatorname{div} \boldsymbol{R}^1.
\end{equation*}

We split $\boldsymbol{R}^1$ into four parts
\begin{equation}
\boldsymbol{R}^1:=\boldsymbol{R}_{\rm lin}+\boldsymbol{R}_{\rm cor}+\boldsymbol{R}_{\rm tem}+\boldsymbol{R}_{\rm osc},
\end{equation}
satisfying
\begin{align*}
\partial_t(\theta_p+\theta_c)&=\operatorname{div}\boldsymbol{R}_{\rm tem},\\
\partial_t\theta_{o}+\operatorname{div}(\theta_p\boldsymbol{w}_p+\boldsymbol{R})&=\operatorname{div}\boldsymbol{R}_{\rm osc},\\
\operatorname{div}(\theta\boldsymbol{u}+\rho\boldsymbol{w})&=\operatorname{div}\boldsymbol{R}_{\rm lin},\\
\operatorname{div}(\theta\boldsymbol{w}_c)+\operatorname{div}(\theta_{o}\boldsymbol{w}_p+\theta_c\boldsymbol{w}_p)&=\operatorname{div}\boldsymbol{R}_{\rm cor}.
\end{align*}

Obviously, $\boldsymbol{R}_{\rm lin},\boldsymbol{R}_{\rm cor}$ can be defined by
\begin{align*}
\boldsymbol{R}_{\rm lin}&:=\theta \boldsymbol{u}+\rho \boldsymbol{w},\\
\boldsymbol{R}_{\rm cor}&:=\theta \boldsymbol{w}_c+(\theta_o+\theta_c)\boldsymbol{w}_p.
\end{align*}

Since
\begin{equation*}
\partial_t(\theta_p+\theta_c)=\partial_t\Big\{\tilde{g}_{\kappa}(\lambda t)\sum\nolimits_{j}\Big(a_j(t,x)\Phi_j^{\mu}(\sigma x)-\fint_{\mathbb{T}^d}a_j(t,x)\Phi_j^{\mu}(\sigma x)\,dx\Big)\Big\},
\end{equation*}
with the help of $\mathcal{B}$, we define $\boldsymbol{R}_{\rm tem}$ by
\begin{equation*}
\boldsymbol{R}_{\rm tem}:=\partial_t\left(\tilde{g}_{\kappa}(\lambda t)\sum\nolimits_{j}\mathcal{B}(a_j,\Phi_j^{\mu}(\sigma x))\right).
\end{equation*}

Now we consider $\boldsymbol{R}_{\rm osc}$. We split $\boldsymbol{R}_{{\rm osc}}$ into three parts
\begin{equation*}
\boldsymbol{R}_{{\rm osc}}:=\boldsymbol{R}_{{\rm osc},x}+\boldsymbol{R}_{{\rm osc},t}+\boldsymbol{R}_{{\rm rem}},
\end{equation*}
where
\begin{align*}
\boldsymbol{R}_{{\rm osc},t}&:=\lambda^{-1} h_{\kappa}(\lambda t)\sum\nolimits_{j}\partial_t(\chi_j^2R_j)\boldsymbol{e}_j,\\
\boldsymbol{R}_{\rm rem}&:=\sum\nolimits_{j}(1-\chi_j^2)R_j\boldsymbol{e}_j.
\end{align*}
Firstly notice
\begin{align*}
\partial_t\theta_{o}&=(\partial_th_{\kappa})(\lambda t)\operatorname{div}\sum\nolimits_{j}\chi_j^2R_j\boldsymbol{e}_j+\lambda^{-1} h_{\kappa}(\lambda t)\partial_t\Big(\operatorname{div}\sum\nolimits_{j}\chi_j^2R_j\boldsymbol{e}_j\Big)\\
&=[1-\bar{g}_{\kappa}\tilde{g}_{\kappa}(\lambda t)]\sum\nolimits_{j}\partial_j(a_jb_j)+\operatorname{div}\boldsymbol{R}_{{\rm osc},t}.
\end{align*}
According to $\operatorname{div}(\Phi_j^{\mu}\boldsymbol{W}_j^{\mu})=0$, and when $i\ne j$, there holds $\Phi_j^{\mu}\boldsymbol{W}_i^{\mu}=0,\boldsymbol{e}_j\cdot\boldsymbol{W}_i^{\mu}=0$. We have
\begin{align*}
\operatorname{div}(\theta_p\boldsymbol{w}_p)&=\bar{g}_{\kappa}\tilde{g}_{\kappa}(\lambda t)\operatorname{div}\sum\nolimits_{j}a_jb_j\Phi_j^{\mu}\boldsymbol{W}_j^{\mu}\\
&=\bar{g}_{\kappa}\tilde{g}_{\kappa}(\lambda t)\sum\nolimits_{j}\partial_j(a_jb_j)\Phi_j^{\mu}\boldsymbol{W}_j^{\mu}\cdot\boldsymbol{e}_j.
\end{align*}
Hence
\begin{align*}
&\partial_t\theta_{o}+\operatorname{div}(\theta_p\boldsymbol{w}_p+\boldsymbol{R})\\
&=\bar{g}_{\kappa}\tilde{g}_{\kappa}(\lambda t)\sum\nolimits_{j}\partial_j(a_jb_j)\big(\Phi_j^{\mu}\boldsymbol{W}_j^{\mu}\cdot\boldsymbol{e}_j-1\big)+\operatorname{div}\boldsymbol{R}_{{\rm osc},t}+\sum\nolimits_{j}\partial_j(a_jb_j)+\operatorname{div}\boldsymbol{R},
\end{align*}

Notice
\begin{equation*}
\sum\nolimits_{j}\partial_j(a_jb_j)+\operatorname{div}\boldsymbol{R}=\operatorname{div}\sum\nolimits_{j}(1-\chi_j^2)R_j\boldsymbol{e}_j=\operatorname{div}\boldsymbol{R}_{\rm rem}.
\end{equation*}
Hence $\boldsymbol{R}_{{\rm osc},x}$ satisfies
\begin{equation*}
\operatorname{div}{\boldsymbol{R}_{{\rm osc},x}}=\bar{g}_{\kappa}\tilde{g}_{\kappa}(\lambda t)\sum\nolimits_{j}\partial_j(a_jb_j)\big(\Phi_j^{\mu}\boldsymbol{W}_j^{\mu}\cdot\boldsymbol{e}_j-1\big).
\end{equation*}
Since
\begin{equation*}
\int_{\mathbb{T}^d}\Phi_j^{\mu}\boldsymbol{W}_j^{\mu}\cdot\boldsymbol{e}_j\,dx=\boldsymbol{e}_j\cdot\boldsymbol{e}_j=1,
\end{equation*}
$\boldsymbol{R}_{{\rm osc},x}$ can be defined by
\begin{equation*}
\boldsymbol{R}_{{\rm osc},x}:=\tilde{g}_{\kappa}\bar{g}_{\kappa}(\lambda t)\sum\nolimits_{j}\mathcal{B}(\partial_j(a_jb_j),\Phi_j^{\mu}\boldsymbol{W}_j^{\mu}(\sigma x)\cdot\boldsymbol{e}_j-1).
\end{equation*}

\subsection{Setting the parameters}\label{sec.4.3}

There are four controllable parameters $\mu,\kappa,\sigma,\lambda$. We will set $\kappa,\sigma,\lambda$ as some positive powers of $\mu$ and then let $\mu$ large enough.

Concentration parameter in space $\mu$: hereafter, we take $\mu\ge\mu_0$ with $\mu_0$ large enough depending on the old solution $(\rho,\boldsymbol{u},\boldsymbol{R})$ and $N,\nu,\delta$ given in Proposition~\ref{prn3.1}, such that all lemmas below hold. Actually, how large $\mu_0$ is and what quantities do $\mu_0$ depend on are inessential as long as $\mu$ is taken to be finite at last, since the most significant matter is that we need to balance the estimates on the perturbations and the new defect field by reasonable setting of $\kappa,\sigma,\lambda$. Recall $\tilde{s}<s'$, the following setting works well.

Concentration parameter in time $\kappa$: setting $\kappa=\mu^{\alpha \frac{s'\tilde{s}}{s'-\tilde{s}}}$
with
\begin{equation}\label{eq-22}
\alpha=1+\frac{(p-1)(d-1)}{2(p+1)p}.
\end{equation}

Oscillation parameter in space $\sigma\in\mathbb{N}$: setting $\sigma=\lfloor\mu^{\beta}\rfloor\le 2\mu^{\beta}$
with
\begin{equation}\label{eq-23}
\beta=\frac{(p-1)(d-1)}{4(p+1)p}.
\end{equation}

Oscillation parameter in time $\lambda\in\mathbb{N}$: setting $\lambda=\lfloor\mu^{\beta/2}\rfloor\le 4\sigma^{1/2}$.

Finally, we choose $N=\lceil\frac{\alpha s'\tilde{s}}{\beta s(s'-\tilde{s})}+\frac{d}{\beta}\rceil$. Where $\lfloor\cdot\rfloor$ is the floor function and $\lceil\cdot\rceil$ is the ceiling function.

\begin{remark}
For the condition \eqref{eq-assum-2}, setting
\begin{align*}
\kappa=\mu^{\alpha \frac{s'\tilde{s}}{s'-\tilde{s}}},\quad \sigma=\lfloor\mu^{\epsilon}\rfloor,\quad \lambda=\lfloor\mu^{\epsilon/2}\rfloor
\end{align*}
with $\epsilon>0$ small enough and
\begin{align*}
\alpha=1+2\beta,\quad\beta=\frac{(p-1)(d-1)}{p}-2\epsilon.
\end{align*}
\end{remark}

\subsection{Estimates on the perturbations}

\begin{lemma}[Estimate on $\theta_p$]\label{lem4-1}
\begin{equation*}
\|\theta_p\|_{L^s_tL^p_x}\lesssim \nu\|\boldsymbol{R}\|_{L^{1}_{t,x}}^{1/p}+\nu C_{*}(\sigma^{-1/p}+\lambda^{-1/s}).
\end{equation*}
In particular, for $\mu$ large enough,
\begin{equation*}
\|\theta_p\|_{L^s_tL^p_x}\lesssim \nu\|\boldsymbol{R}\|_{L^{1}_{t,x}}^{1/p}.
\end{equation*}
\end{lemma}
\begin{proof}
\begin{equation*}
\|\theta_p(t)\|_{L^p}\le \vert \tilde{g}_{\kappa}(\lambda t)\vert \sum\nolimits_{j}\|a_j(t)\Phi_j^{\mu}(\sigma\cdot)\|_{L^p}.
\end{equation*}
Since $a_j(t,\cdot)$ is smooth on $\mathbb{T}^d$, by Lemma~\ref{lem2.3} and \eqref{eq-18}, we have
\begin{equation*}
\|a_j(t)\Phi_j^{\mu}(\sigma\cdot)\|_{L^p}\le \|a_j(t)\|_{L^p}\|\Phi_j^{\mu}\|_{L^p}+\nu C_{*}\sigma^{-1/p}\|\Phi_j^{\mu}\|_{L^p}.
\end{equation*}
Combining \eqref{eq-6} and \eqref{eq-17}, we obtain
\begin{equation*}
\|\theta_p(t)\|_{L^p}\lesssim \nu\vert \tilde{g}_{\kappa}(\lambda t)\vert \sum\nolimits_{j}(\|\widetilde{R}_j\|_{L^{1}_{t,x}}^{\frac{1}{p}-\frac{1}{s}}\|\widetilde{R}_j(t)\|_{L^1}^{\frac{1}{s}}+C_{*}\sigma^{-1/p}).
\end{equation*}
Now we take $L^s$ in time to obtain
\begin{equation*}
\|\theta_p\|_{L^s_tL^p_x}\lesssim \nu\sum\nolimits_{j}\|\widetilde{R}_j\|_{L^{1}_{t,x}}^{\frac{1}{p}-\frac{1}{s}}\Big(\int_{[0,1]}\vert \tilde{g}_{\kappa}(\lambda t)\vert ^s\|\widetilde{R}_j(t)\|_{L^1}\,dt\Big)^{1/s}+\nu \|\tilde{g}_{\kappa}\|_{L^s_t}C_{*}\sigma^{-1/p}.
\end{equation*}
Applying Lemma~\ref{lem2.3} in time once again gives
\begin{equation*}
\int_{[0,1]}\vert \tilde{g}_{\kappa}(\lambda t)\vert ^s\|\widetilde{R}_j(t)\|_{L^1}\,dt\le \|\widetilde{R}_j\|_{L^{1}_{t,x}}\|\tilde{g}_{\kappa}\|^s_{L^s_t}+C_{*}\lambda^{-1}
\end{equation*}
Finally, by \eqref{eq-12} we have
\begin{align*}
\|\theta_p\|_{L^s_tL^p_x}&\lesssim \nu\sum\nolimits_{j}\|\widetilde{R}_j\|_{L^{1}_{t,x}}^{1/p}\|\tilde{g}_{\kappa}\|_{L^s_t}+\nu C_{*}(\sigma^{-1/p}+\lambda^{-1/s})\\
&\lesssim \nu\|\boldsymbol{R}\|_{L^{1}_{t,x}}^{1/p}+\nu C_{*}(\sigma^{-1/p}+\lambda^{-1/s}).
\end{align*}
\end{proof}

\begin{lemma}[Estimate on $\theta_c$]\label{lem4-2}
For $N=\lceil\frac{\alpha s'\tilde{s}}{\beta s(s'-\tilde{s})}+\frac{d}{\beta}\rceil$, we have
\begin{equation*}
\|\theta_c\|_{L^{\infty}_t}\le C_{*}\nu\mu^{-\frac{d-1}{p'}-\frac{d-1}{2}}.
\end{equation*}
In particular, for $\mu$ large enough,
\begin{equation*}
\|\theta_c\|_{L^s_tL^p_x}\le\nu\|\boldsymbol{R}\|_{L^{1}_{t,x}}^{1/p}.
\end{equation*}
\end{lemma}
\begin{proof}
Since
\begin{equation*}
\theta_c=\tilde{g}_{\kappa}(\lambda t)\sum\nolimits_{j}\fint_{\mathbb{T}^d}a_j(t,x)\Phi_j^{\mu}(\sigma x)\,dx,
\end{equation*}
by Lemma~\ref{lem2.4} and notice $N>\frac{\alpha s'\tilde{s}}{\beta s(s'-\tilde{s})}+\frac{d}{\beta}$, we have
\begin{align*}
\|\theta_c\|_{L^{\infty}_t}&\le \sigma^{-N}\|\tilde{g}_{\kappa}\|_{L^{\infty}}\sum\nolimits_{j}\|a_j\|_{C^N}\|\Phi_j^{\mu}\|_{L^2}\\
\text{\scriptsize by Proposition~\ref{prn2.5},\eqref{eq-12},\eqref{eq-18}}&\le C_{*}\nu\sigma^{-N}\kappa^{\frac{1}{s}}\mu^{\frac{d-1}{p}-\frac{d-1}{2}}\\
\text{\scriptsize by \eqref{eq-22},\eqref{eq-23}}&\le C_{*}\nu\mu^{\frac{d-1}{p}-\frac{d-1}{2}-d}.
\end{align*}
\end{proof}

\begin{lemma}[Estimate on $\theta_o$]\label{lem4-3}
\begin{equation*}
\|\theta_o\|_{L^{\infty}_{t,x}}\le C_{*}\lambda^{-1}.
\end{equation*}
In particular, for $\mu$ large enough,
\begin{equation*}
\|\theta_o\|_{L^s_tL^p_x}\le\nu\|\boldsymbol{R}\|_{L^{1}_{t,x}}^{1/p}.
\end{equation*}
\end{lemma}
\begin{proof}
We control $\|\nabla(\chi_j^2R_j)\|_{L^{\infty}_{t,x}}$ simply by a constant $C_*$, then by \eqref{eq-13} we obtain
\begin{equation*}
\|\theta_o\|_{L^{\infty}_{t,x}}\le\lambda^{-1}\|h_{\kappa}\|_{L^{\infty}_t}\sum\nolimits_{j}\|e_j\cdot\nabla(\chi_j^2R_j)\|_{L^{\infty}_{t,x}}\le C_{*}\lambda^{-1}.
\end{equation*}
\end{proof}

\begin{lemma}[Estimate on $\boldsymbol{w}_p$ with $L^{s^{\prime}}_tL^{p^{\prime}}_x$ norm]\label{}
\begin{equation*}
\|\boldsymbol{w}_p\|_{L^{s^{\prime}}_tL^{p^{\prime}}_x}\lesssim \sum\nolimits_{j}\nu^{-1}\|\widetilde{R}_j\|_{L^{1}_{t,x}}^{1-\frac{1}{p}}+\nu^{-1}C_{*}(\lambda^{-1/s'}+\sigma^{-1/p'})
\end{equation*}
In particular, for $\mu$ large enough,
\begin{equation*}
\|\boldsymbol{w}_p\|_{L^{s^{\prime}}_tL^{p^{\prime}}_x}\lesssim\nu^{-1}\|\boldsymbol{R}\|_{L^{1}_{t,x}}^{1/p^{\prime}}.
\end{equation*}
\end{lemma}
\begin{proof}
By Lemma~\ref{lem2.3} and \eqref{eq-18},
\begin{align*}
\|\boldsymbol{w}_p(t)\|_{L^{p'}}&\le\vert \bar{g}_{\kappa}(\lambda t)\vert \sum\nolimits_{j}\|b_j(t)\boldsymbol{W}_j^{\mu}(\sigma\cdot)\|_{L^{p'}}\\
&\le\vert \bar{g}_{\kappa}(\lambda t)\vert \sum\nolimits_{j}(\|b_j(t)\|_{L^{p'}}\|\boldsymbol{W}_j^{\mu}\|_{L^{p'}}+\nu^{-1}C_{*}\sigma^{-1/p'}\|\boldsymbol{W}_j^{\mu}\|_{L^{p'}}).
\end{align*}
By Proposition~\ref{prn2.5} and \eqref{eq-12}, we have $\|\boldsymbol{W}_j^{\mu}\|_{L^{p'}}\lesssim 1$ and $\|\bar{g}_{\kappa}(\lambda \cdot)\|_{L^{s'}_t}=\|\bar{g}_{\kappa}\|_{L^{s'}_t}\lesssim 1$. Hence
\begin{align*}
\|\boldsymbol{w}_p\|_{L^{s^{\prime}}_tL^{p^{\prime}}_x}&\lesssim \sum\nolimits_{j}\Big(\int_{[0,1]}\vert \bar{g}_{\kappa}(\lambda t)\vert ^{s'}\|b_j(t)\|_{L^{p'}}^{s'}\,dt\Big)^{1/s'}+\nu^{-1}C_{*}\sigma^{-1/p'}\\
\text{\scriptsize by \eqref{eq-17}}&\lesssim \sum\nolimits_{j}\nu^{-1}\|\widetilde{R}_j\|_{L^{1}_{t,x}}^{\frac{1}{s}-\frac{1}{p}}\Big(\int_{[0,1]}\vert \bar{g}_{\kappa}(\lambda t)\vert ^{s'}\|\widetilde{R}_j(t)\|_{L^1}\,dt\Big)^{\frac{1}{s'}}+\nu^{-1}C_{*}\sigma^{-\frac{1}{p'}}\\
\text{\scriptsize by Lemma~\ref{lem2.3}}&\lesssim \sum\nolimits_{j}\nu^{-1}\|\widetilde{R}_j\|_{L^{1}_{t,x}}^{\frac{1}{s}-\frac{1}{p}}\Big(\|\widetilde{R}_j\|_{L^{1}_{t,x}}\|\bar{g}_{\kappa}\|^{s'}_{L^{s'}_t}+C_{*}\lambda^{-1}\Big)^{\frac{1}{s'}}+\nu^{-1}C_{*}\sigma^{-\frac{1}{p'}}\\
&\lesssim \sum\nolimits_{j}\nu^{-1}\|\widetilde{R}_j\|_{L^{1}_{t,x}}^{1-\frac{1}{p}}+\nu^{-1}C_{*}(\lambda^{-1/s'}+\sigma^{-1/p'})
\end{align*}
\end{proof}

\begin{lemma}[Estimate on $\boldsymbol{w}_p$ with $L^{\tilde{s}}_tW^{1,q}_x$ norm]\label{lem4-5}
\begin{equation*}
\|\boldsymbol{w}_p\|_{L^{\tilde{s}}_tW^{1,q}_x}\le\nu^{-1}C_{*}\mu^{\gamma_1},
\end{equation*}
for some $\gamma_1<0$. In particular, for $\mu$ large enough, we have
\begin{equation*}
\|\boldsymbol{w}_p\|_{L^{\tilde{s}}_tW^{1,q}_x}\le\frac{\delta}{20}.
\end{equation*}
\end{lemma}
\begin{remark}
$s<\infty$ and the first condition of \eqref{eq-assum} are required in this estimate.
\end{remark}
\begin{proof}
\begin{align*}
\|\boldsymbol{w}_p(t)\|_{W^{1,q}_x}&\le \vert \bar{g}_{\kappa}(\lambda t)\vert \sum\nolimits_{j}\|b_j(t)\boldsymbol{W}_j^{\mu}(\sigma\cdot)\|_{W^{1,q}_x}\\
&\le \vert \bar{g}_{\kappa}(\lambda t)\vert \sum\nolimits_{j}\|b_j\|_{C^1}\|\boldsymbol{W}_j^{\mu}(\sigma\cdot)\|_{W^{1,q}_x}\\
\text{\scriptsize by \eqref{eq-18}}&\le \nu^{-1}C_{*}\vert \bar{g}_{\kappa}(\lambda t)\vert \sum\nolimits_{j}(\|\boldsymbol{W}_j^{\mu}\|_{q}+\sigma\|\nabla\boldsymbol{W}_j^{\mu}\|_{L^q})\\
\text{\scriptsize by Proposition~\ref{prn2.5}}&\le \nu^{-1}C_{*}\vert \bar{g}_{\kappa}(\lambda t)\vert \sigma\mu^{1+\frac{d-1}{p'}-\frac{d-1}{q}}.
\end{align*}
Hence
\begin{align*}
\|\boldsymbol{w}_p\|_{L^{\tilde{s}}_tW^{1,q}_x}&\le\nu^{-1}C_{*}\sigma\mu^{1+\frac{d-1}{p'}-\frac{d-1}{q}}\|\bar{g}_{\kappa}\|_{L^{\tilde{s}}([0,1])}\\
\text{\scriptsize by \eqref{eq-12}}&\le\nu^{-1}C_{*}\sigma\mu^{1+\frac{d-1}{p'}-\frac{d-1}{q}}\kappa^{\frac{\tilde{s}-s'}{s'\tilde{s}}}\\
\text{\scriptsize by \eqref{eq-22},\eqref{eq-23}}&\le\nu^{-1}C_{*}\mu^{1+\frac{d-1}{p'}-\frac{d-1}{q}+\beta-\alpha}\\
&\le\nu^{-1}C_{*}\mu^{\gamma_1},
\end{align*}
where $\gamma_1=(d-1)[1-\frac{p-1}{4(p+1)p}-(\frac{1}{p}+\frac{1}{q})]<0$ by the assumption \eqref{eq-assum}.
\end{proof}

\begin{lemma}[Estimate on $\boldsymbol{w}_c$]\label{lem4-6}
\begin{align*}
\|\boldsymbol{w}_c\|_{L^{s'}_tL^{p^{\prime}}_x}&\le\nu^{-1}C_{*}\sigma^{-1},\\
\|\boldsymbol{w}_c\|_{L^{\tilde{s}}_tW^{1,q}_x}&\le\nu^{-1}C_{*}\mu^{\gamma_2}
\end{align*}
for some $\gamma_2<0$. In particular, for $\mu$ large enough, we have
\begin{align*}
\|\boldsymbol{w}_c\|_{L^{s^{\prime}}_tL^{p^{\prime}}_x}&\le\nu^{-1}\|\boldsymbol{R}\|_{L^{1}_{t,x}}^{1/p^{\prime}},\\
\|\boldsymbol{w}_c\|_{L^{\tilde{s}}_tW^{1,q}_x}&\le\frac{\delta}{20}.
\end{align*}
\end{lemma}
\begin{remark}
$s<\infty$ and the first condition of \eqref{eq-assum} are required in this estimate.
\end{remark}
\begin{proof}
By Lemma~\ref{lem2.2}, Proposition~\ref{prn2.5} and \eqref{eq-18}, we have
\begin{align*}
\|\boldsymbol{w}_c(t)\|_{L^{p'}}&\lesssim\vert \bar{g}_{\kappa}(\lambda t)\vert \sum\nolimits_{j}\|b_j\|_{C^2}\|\mathcal{R}\boldsymbol{W}_j^{\mu}(\sigma\cdot)\|_{L^{p'}}\\
&\le\nu^{-1}C_{*}\vert \bar{g}_{\kappa}(\lambda t)\vert \sum\nolimits_{j}\sigma^{-1}\|\boldsymbol{W}_j^{\mu}\|_{L^{p'}}\\
&\le \nu^{-1}C_{*}\vert \bar{g}_{\kappa}(\lambda t)\vert \sigma^{-1}.
\end{align*}
Hence by \eqref{eq-12}, we have
\begin{equation*}
\|\boldsymbol{w}_c\|_{L^{s'}_tL^{p^{\prime}}_x}\le\nu^{-1}C_{*}\|\bar{g}_{\kappa}\|_{L^{s'}_t}\sigma^{-1}\le\nu^{-1}C_{*}\sigma^{-1}.
\end{equation*}

By Lemma~\ref{lem2.1}, Proposition~\ref{prn2.5} and the definition of $\mathcal{B}$, we have
\begin{align*}
\|\boldsymbol{w}_c(t)\|_{W^{1,q}_x}&\lesssim\vert \bar{g}_{\kappa}(\lambda t)\vert \sum\nolimits_{j}\|\mathcal{B}(\partial_jb_j,\boldsymbol{W}_j^{\mu}(\sigma\cdot)\cdot\boldsymbol{e}_j)\|_{W^{1,q}_x}\\
&\lesssim\vert \bar{g}_{\kappa}(\lambda t)\vert \sum\nolimits_{j}(\|\partial_jb_j\mathcal{R}(\boldsymbol{W}_j^{\mu}(\sigma\cdot)\cdot\boldsymbol{e}_j)\|_{W^{1,q}_x}\\
&\quad+\|\mathcal{R}(\nabla(\partial_jb_j)\cdot(\mathcal{R}(\boldsymbol{W}_j^{\mu}(\sigma\cdot)\cdot\boldsymbol{e}_j)))\|_{W^{1,q}_x})\\
&\lesssim\vert \bar{g}_{\kappa}(\lambda t)\vert \sum\nolimits_{j}(\|b_j\|_{C^2}\|\boldsymbol{W}_j^{\mu}(\cdot)\cdot\boldsymbol{e}_j\|_{W^{1,q}_x}+\|b_j\|_{C^3}\|\boldsymbol{W}_j^{\mu}(\cdot)\cdot\boldsymbol{e}_j\|_{W^{1,q}_x})\\
\text{\scriptsize by \eqref{eq-18}}&\le\nu^{-1}C_{*}\vert \bar{g}_{\kappa}(\lambda t)\vert \mu^{1+\frac{d-1}{p'}-\frac{d-1}{q}}
\end{align*}
Hence
\begin{align*}
\|\boldsymbol{w}_c\|_{L^{\tilde{s}}_tW^{1,q}_x}&\le\nu^{-1}C_{*}\mu^{1+\frac{d-1}{p'}-\frac{d-1}{q}}\|\bar{g}_{\kappa}\|_{L^{\tilde{s}}_t}\\
\text{\scriptsize by \eqref{eq-12}}&\le\nu^{-1}C_{*}\mu^{1+\frac{d-1}{p'}-\frac{d-1}{q}}\kappa^{\frac{\tilde{s}-s'}{s'\tilde{s}}}\\
\text{\scriptsize by \eqref{eq-22}}&\le\nu^{-1}C_{*}\mu^{\gamma_2}
\end{align*}
where $\gamma_2=(d-1)[1-\frac{p-1}{2(p+1)p}-(\frac{1}{p}+\frac{1}{q})]<0$ by the assumption \eqref{eq-assum}.
\end{proof}

\begin{lemma}\label{lem4-7}
For $N=\lceil\frac{\alpha s'\tilde{s}}{\beta s(s'-\tilde{s})}+\frac{d}{\beta}\rceil$ and $\mu$ large enough, we have
\begin{equation*}
\Big\vert \int_{\mathbb{T}^d}\theta(t,x)\phi(x)\,dx\Big\vert \le\delta\|\phi\|_{C^N}\quad\forall t\in[0,1],\,\forall\phi\in C^{\infty}(\mathbb{T}^d).
\end{equation*}
Moreover $\operatorname{supp}_t\theta\in I_r$.
\end{lemma}
\begin{proof}
By the definition, it is obvious that $\operatorname{supp}_t\theta\in I_r$. Notice
\begin{align*}
&\Big\vert \int_{\mathbb{T}^d}\theta(t,x)\phi(x)\,dx\Big\vert\\
&\le\Big\vert \int_{\mathbb{T}^d}\theta_p(t,x)\phi(x)\,dx\Big\vert +\Big\vert \int_{\mathbb{T}^d}\theta_c(t,x)\phi(x)\,dx\Big\vert +\Big\vert \int_{\mathbb{T}^d}\theta_o(t,x)\phi(x)\,dx\Big\vert.
\end{align*}
By Lemma~\ref{lem2.4}, Proposition~\ref{prn2.5} and notice $N>\frac{\alpha s'\tilde{s}}{\beta s(s'-\tilde{s})}+\frac{d}{\beta}$, we have
\begin{align*}
\Big\vert \int_{\mathbb{T}^d}\theta_p(t,x)\phi(x)\,dx\Big\vert &\le\sigma^{-N}\|\tilde{g}_{\kappa}\|_{L^{\infty}}\sum\nolimits_{j}\|a_j\phi\|_{C^N}\|\Phi_j^{\mu}\|_{L^2}\\
\text{\scriptsize by \eqref{eq-12},\eqref{eq-18}}&\le \nu C_{*}\sigma^{-N}\kappa^{\frac{1}{s}}\mu^{\frac{d-1}{p}-\frac{d-1}{2}}\|\phi\|_{C^N}\\
\text{\scriptsize by \eqref{eq-22},\eqref{eq-23}}&\le \nu C_{*}\mu^{-\frac{d-1}{p'}-\frac{d-1}{2}}\|\phi\|_{C^N}.
\end{align*}
By Lemma~\ref{lem4-2}, we have
\begin{equation*}
\Big\vert \int_{\mathbb{T}^d}\theta_c(t,x)\phi(x)\,dx\Big\vert \le\|\theta_c\|_{L^{\infty}_{t,x}}\|\phi\|_{L^{\infty}}\le \nu C_{*}\mu^{-\frac{d-1}{p'}-\frac{d-1}{2}}\|\phi\|_{L^{\infty}}.
\end{equation*}
Similarly, by Lemma~\ref{lem4-3}, we have
\begin{equation*}
\Big\vert \int_{\mathbb{T}^d}\theta_o(t,x)\phi(x)\,dx\Big\vert \le\|\theta_o\|_{L^{\infty}_{t,x}}\|\phi\|_{L^{\infty}}\le C_{*}\lambda^{-1}\|\phi\|_{L^{\infty}}.
\end{equation*}
In conclusion, we have
\begin{equation*}
\Big\vert \int_{\mathbb{T}^d}\theta(t,x)\phi(x)\,dx\Big\vert \le \nu C_{*}\Big(\mu^{-\frac{d-1}{p'}-\frac{d-1}{2}}+\nu^{-1}\lambda^{-1}\Big)\|\phi\|_{C^N}.
\end{equation*}
\end{proof}

\subsection{Estimates on the new defect field}

\begin{lemma}[Estimate on $\boldsymbol{R}_{\rm tem}$]\label{lem-924-4.8}
For $\mu$ large enough, we have
\begin{equation}\label{eq-25}
\|\boldsymbol{R}_{\rm tem}\|_{L^{1}_{t,x}}\le\frac{\delta}{10}.
\end{equation}
\end{lemma}
\begin{remark}
$p>1$ (for the C-Z inequality \eqref{eq-5}) and the second condition of \eqref{eq-assum} are required in this estimate.
\end{remark}
\begin{proof}
\begin{align*}
\boldsymbol{R}_{\rm tem}&=\partial_t(\tilde{g}_{\kappa}(\lambda t))\sum\nolimits_{j}\mathcal{B}(a_j,\Phi_j^{\mu}(\sigma x))+\tilde{g}_{\kappa}(\lambda t)\sum\nolimits_{j}\mathcal{B}(\partial_ta_j,\Phi_j^{\mu}(\sigma x))\\
&=:\boldsymbol{R}_{\rm tem,1}+\boldsymbol{R}_{\rm tem,2}.
\end{align*}
For $\boldsymbol{R}_{\rm tem,1}$, by Lemma~\ref{lem2.2} and notice $\mathcal{R}[\Phi_j^{\mu}(\sigma\cdot)](x)=\sigma^{-1}(\mathcal{R}\Phi_j^{\mu})(\sigma x)$, we have
\begin{align*}
\|\boldsymbol{R}_{\rm tem,1}(t)\|_{L^1}&\le\lambda(\partial_t\tilde{g}_{\kappa})(\lambda t)\sum\nolimits_{j}\|\mathcal{B}(a_j,\Phi_j^{\mu}(\sigma \cdot))\|_{L^1}\\
\text{\scriptsize by \eqref{eq-18}}&\le\lambda(\partial_t\tilde{g}_{\kappa})(\lambda t)\sum\nolimits_{j}C_{*}\|\mathcal{R}(\Phi_j^{\mu}(\sigma \cdot))\|_{L^1}\\
&\le\lambda(\partial_t\tilde{g}_{\kappa})(\lambda t)\sum\nolimits_{j}C_{*}\sigma^{-1}\|\mathcal{R}\Phi_j^{\mu}\|_{L^1}\\
\text{\scriptsize by \eqref{eq-7}}&\le\lambda\sigma^{-1}(\partial_t\tilde{g}_{\kappa})(\lambda t)\sum\nolimits_{j}C_{*}\|\mathcal{R}(\operatorname{div}\boldsymbol{\Omega}_j^{\mu})\|_{L^1}\\
\text{\scriptsize by \eqref{eq-5} and $p>1$}&\le\lambda\sigma^{-1}(\partial_t\tilde{g}_{\kappa})(\lambda t)\sum\nolimits_{j}C_{*}\|\boldsymbol{\Omega}_j^{\mu}\|_{L^{\frac{p+1}{2}}}.
\end{align*}
(When under the conditions \eqref{eq-assum-2}, we apply \eqref{eq-5} to the last inequality with $r>1$ sufficiently close to 1). Hence by \eqref{eq-12} and Proposition~\ref{prn2.5}
\begin{align*}
\|\boldsymbol{R}_{\rm tem,1}\|_{L^{1}_{t,x}}&\le\lambda\sigma^{-1}\|\partial_t\tilde{g}_{\kappa}\|_{L^1_t}\sum\nolimits_{j}C_{*}\|\boldsymbol{\Omega}_j^{\mu}\|_{L^{\frac{p+1}{2}}}\\
&\le C_{*}\lambda\sigma^{-1}\kappa^{1/s}\mu^{-1+\frac{d-1}{p}-\frac{2(d-1)}{p+1}}\\
\text{\scriptsize by \eqref{eq-22},\eqref{eq-23}}&\le C_{*}\sigma^{-1/2}\mu^{1+4\beta}\mu^{-1-4\beta}\\
&\le C_{*}\mu^{-\beta/2}.
\end{align*}
where we have used $\kappa^{1/s}\le \mu^{1+4\beta}$ by the second assumption of \eqref{eq-assum}. Hence take $\mu$ large enough, we have
\begin{equation}\label{eq-26}
\|\boldsymbol{R}_{\rm tem,1}\|_{L^{1}_{t,x}}\le\frac{\delta}{20}.
\end{equation}

For $\boldsymbol{R}_{\rm tem,2}$, we have by Lemma~\ref{lem2.2}
\begin{align*}
\|\boldsymbol{R}_{\rm tem,2}(t)\|_{L^1}&\le \tilde{g}_{\kappa}(\lambda t)\sum\nolimits_{j}\|\mathcal{B}(\partial_ta_j,\Phi_j^{\mu}(\sigma \cdot))\|_{L^1}\\
&\le \tilde{g}_{\kappa}(\lambda t)\sum\nolimits_{j}C_{*}\|\mathcal{R}(\Phi_j^{\mu}(\sigma \cdot))\|_{L^1}\\
&\le \sigma^{-1}\tilde{g}_{\kappa}(\lambda t)\sum\nolimits_{j}C_{*}\|\Phi_j^{\mu}\|_{L^1}.
\end{align*}
Now by Proposition~\ref{prn2.5} and \eqref{eq-12}, we have
\begin{equation*}
\|\boldsymbol{R}_{\rm tem,2}\|_{L^{1}_{t,x}}\le C_{*}\sigma^{-1}\kappa^{-\frac{1}{s'}}\mu^{-\frac{d-1}{p'}},
\end{equation*}
take $\mu$ large enough, we have
\begin{equation}\label{eq-27}
\|\boldsymbol{R}_{\rm tem,2}\|_{L^{1}_{t,x}}\le\frac{\delta}{20}.
\end{equation}
Adding \eqref{eq-26},\eqref{eq-27} results in \eqref{eq-25}.
\end{proof}

\begin{lemma}[Estimate on $\boldsymbol{R}_{\rm lin}$]\label{}
For $\mu$ large enough, we have
\begin{equation*}
\|\boldsymbol{R}_{\rm lin}\|_{L^{1}_{t,x}}\le\frac{\delta}{10}.
\end{equation*}
\end{lemma}
\begin{proof}
\begin{align*}
\|\boldsymbol{R}_{\rm lin}\|_{L^{1}_{t,x}}&\le\|\theta\|_{L^{1}_{t,x}} \|\boldsymbol{u}\|_{L^{\infty}_{t,x}}+ \|\boldsymbol{w}\|_{L^{1}_{t,x}}\|\rho\|_{L^{\infty}_{t,x}}.
\end{align*}
Notice by Proposition~\ref{prn2.5} and \eqref{eq-12}
\begin{align*}
\|\theta_p+\theta_c\|_{L^{1}_{t,x}}&\le 2\sum\nolimits_{j}\|\tilde{g}_{\kappa}(\lambda t)a_j(t,x)\Phi_j^{\mu}(\sigma\cdot)\|_{L^1_{t,x}}\\
\text{\scriptsize by \eqref{eq-18}}&\le \nu C_{*}\sum\nolimits_{j}\|\tilde{g}_{\kappa}\|_{L^1_t}\|\Phi_j^{\mu}\|_{L^1}\\
&\le \nu C_{*}\kappa^{-1/s'}\mu^{\frac{(1-p)(d-1)}{p}}.
\end{align*}
By Lemma~\ref{lem4-3}
\begin{equation*}
\|\theta_o\|_{L^{1}_{t,x}}\le\|\theta_o\|_{L^{\infty}_{t,x}}\le C_{*}\lambda^{-1}.
\end{equation*}
By Lemma~\ref{lem4-5} and \ref{lem4-6}
\begin{align*}
\|\boldsymbol{w}\|_{L^{1}_{t,x}}&\le \|\boldsymbol{w}_p\|_{L^{\tilde{s}}_{t}W^{1,q}_x}+\|\boldsymbol{w}_c\|_{L^{\tilde{s}}_{t}W^{1,q}_x}\\
&\le \nu^{-1}C_{*}\mu^{\gamma_1}+\nu^{-1}C_{*}\mu^{\gamma_2}.
\end{align*}
In conclusion, we have
\begin{equation*}
\|\boldsymbol{R}_{\rm lin}\|_{L^{1}_{t,x}}\le C_{*}\Big(\nu \kappa^{-1/s'}\mu^{\frac{(1-p)(d-1)}{p}}+\lambda^{-1}+\nu^{-1}\mu^{\gamma_1}+\nu^{-1}\mu^{\gamma_2}\Big),
\end{equation*}
take $\mu$ large enough, we have
\begin{equation*}
\|\boldsymbol{R}_{\rm lin}\|_{L^{1}_{t,x}}\le\frac{\delta}{10}.
\end{equation*}
\end{proof}

\begin{lemma}[Estimate on $\boldsymbol{R}_{\rm cor}$]\label{}
For $\mu$ large enough, we have
\begin{equation*}
\|\boldsymbol{R}_{\rm cor}\|_{L^{1}_{t,x}}\le\frac{\delta}{10}.
\end{equation*}
\end{lemma}
\begin{proof}
\begin{equation*}
\|\boldsymbol{R}_{\rm cor}\|_{L^{1}_{t,x}}\le\|\theta\|_{L^{s}_{t}L^{p}_x} \|\boldsymbol{w}_c\|_{L^{s'}_{t}L^{p'}_x}+\|\theta_o+\theta_c\|_{L^{s}_{t}L^{p}_x}\|\boldsymbol{w}\|_{L^{s'}_{t}L^{p'}_x}.
\end{equation*}
Notice by Lemma~\ref{lem4-1}-\ref{lem4-3} and Lemma~\ref{lem4-6}, we have
\begin{equation*}
\|\theta\|_{L^{s}_{t}L^{p}_x} \|\boldsymbol{w}_c\|_{L^{s'}_{t}L^{p'}_x}\le \|\boldsymbol{R}\|_{L^{1}_{t,x}}^{1/p}C_{*}\sigma^{-1}.
\end{equation*}
By Lemma~\ref{lem4-2}-\ref{lem4-6}, we have
\begin{equation*}
\|\theta_o+\theta_c\|_{L^{s}_{t}L^{p}_x}\|\boldsymbol{w}\|_{L^{s'}_{t}L^{p'}_x}\le C_{*}(\lambda^{-1}+\nu\mu^{-\frac{d-1}{p'}-\frac{d-1}{2}})\nu^{-1}\|\boldsymbol{R}\|_{L^{1}_{t,x}}^{1/p^{\prime}}.
\end{equation*}
Hence 
\begin{equation*}
\|\boldsymbol{R}_{\rm cor}\|_{L^{1}_{t,x}}\le C_{*}(\sigma^{-1}+\nu^{-1}\lambda^{-1}+\mu^{-\frac{d-1}{p'}-\frac{d-1}{2}}),
\end{equation*}
take $\mu$ large enough, we have
\begin{equation*}
\|\boldsymbol{R}_{\rm cor}\|_{L^{1}_{t,x}}\le\frac{\delta}{10}.
\end{equation*}
\end{proof}

\begin{lemma}[Estimate on $\boldsymbol{R}_{{\rm osc},x}$]\label{}
For $\mu$ large enough, we have
\begin{equation*}
\|\boldsymbol{R}_{{\rm osc},x}\|_{L^{1}_{t,x}}\le\frac{\delta}{10}.
\end{equation*}
\end{lemma}
\begin{proof}
Denote $\Theta_j(\sigma x):=\Phi_j^{\mu}(\sigma x)\boldsymbol{W}_j^{\mu}(\sigma x)\cdot\boldsymbol{e}_j-1$. By Lemma~\ref{lem2.1}-\ref{lem2.2} and notice $\mathcal{R}[\Theta_j(\sigma\cdot)](x)=\sigma^{-1}(\mathcal{R}\Theta_j)(\sigma x)$, we have
\begin{align*}
\|\boldsymbol{R}_{{\rm osc},x}(t)\|_{L^1}&\le\vert \tilde{g}_{\kappa}\bar{g}_{\kappa}(\lambda t)\vert \sum\nolimits_{j}\|\mathcal{B}(\partial_j(a_jb_j),\Theta_j(\sigma\cdot))\|_{L^1}\\
&\le\vert \tilde{g}_{\kappa}\bar{g}_{\kappa}(\lambda t)\vert \sum\nolimits_{j}\|a_jb_j\|_{C^2}\|\mathcal{R}(\Theta_j(\sigma\cdot))\|_{L^1}\\
\text{\scriptsize by \eqref{eq-18}}&\le C_{*}\vert \tilde{g}_{\kappa}\bar{g}_{\kappa}(\lambda t)\vert \sum\nolimits_{j}\sigma^{-1}\|\Theta_j\|_{L^1}\\
&\le C_{*}\vert \tilde{g}_{\kappa}\bar{g}_{\kappa}(\lambda t)\vert \sum\nolimits_{j}\sigma^{-1}\Big(\|\Phi_j^{\mu}\|_{L^p}\|\boldsymbol{W}_j^{\mu}\|_{L^{p'}}+1\Big).
\end{align*}
By Proposition~\ref{prn2.5}, we have $\|\Phi_j^{\mu}\|_{L^p}\|\boldsymbol{W}_j^{\mu}\|_{L^{p'}}\lesssim 1$. Hence by \eqref{eq-12}, we finally obtain
\begin{equation*}
\|\boldsymbol{R}_{{\rm osc},x}\|_{L^1_{t,x}}\le C_{*}\sigma^{-1}\|\tilde{g}_{\kappa}\bar{g}_{\kappa}\|_{L^1_t}\le C_{*}\sigma^{-1}.
\end{equation*}
\end{proof}

\begin{lemma}[Estimate on $\boldsymbol{R}_{{\rm osc},t}$]\label{}
For $\mu$ large enough, we have
\begin{equation*}
\|\boldsymbol{R}_{{\rm osc},t}\|_{L^{1}_{t,x}}\le\frac{\delta}{10}.
\end{equation*}
\end{lemma}
\begin{proof}
We control $\|\partial_t(\chi_j^2R_j)\|_{L^{\infty}_{t,x}}$ simply by a constant $C_*$, then by Lemma~\ref{lem2.1} we obtain
\begin{align*}
\|\boldsymbol{R}_{{\rm osc},t}\|_{L^{1}_{t,x}}&\le\lambda^{-1}\| h_{\kappa}(\lambda t)\sum\nolimits_{j}\partial_t(\chi_j^2R_j)\|_{L^1_{t,x}}\\
&\le C\lambda^{-1}\| h_{\kappa}\|_{L^1_t}\sum\nolimits_{j}\|\partial_t(\chi_j^2R_j)\|_{L^{\infty}_{t,x}}\\
\text{\scriptsize by \eqref{eq-13}}&\le C_{*}\lambda^{-1}.
\end{align*}
\end{proof}

\begin{lemma}[Estimate on $\boldsymbol{R}_{\rm rem}$]\label{lem4-13}
\begin{equation*}
\|\boldsymbol{R}_{\rm rem}\|_{L^{1}_{t,x}}\le\frac{\delta}{2}.
\end{equation*}
\end{lemma}
\begin{proof}
According to the definition \eqref{eq-14} of $\chi_j$ and the choice \eqref{eq-15} of $r$, we have
\begin{align*}
\|\boldsymbol{R}_{\rm rem}\|_{L^{1}_{t,x}}&\le\sum\nolimits_{j}\|(1-\chi_j^2)R_j\boldsymbol{e}_j\|_{L^{1}_{t,x}}\\
&\le\sum\nolimits_{j}\int_{\vert R_j\vert \le\frac{\delta}{4d}}(1-\chi_j^2)\vert R_j\vert \,dxdt+\int_{t\in I^c_{r}}(1-\chi_j^2)\vert R_j\vert \,dxdt\\
&\le d\times(\frac{\delta}{4d}+2r\|\boldsymbol{R}\|_{L^{\infty}_{t,x}})=\frac{\delta}{2}.
\end{align*}
\end{proof}

\subsection{Conclusion}

\noindent\textbf{Proof of Proposition~\ref{prn3.1}.} Combine Lemma~\ref{lem4-1}-\ref{lem4-13}, there exists $\mu_0\ge 8d$ large enough depending on the old solution $(\rho,\boldsymbol{u},\boldsymbol{R})$ and $N,\nu,\delta$ given in Proposition~\ref{prn3.1}, such that for $\mu\ge\mu_0$, we have
\begin{align*}
\|\theta\|_{L^s_tL^p_x}&\lesssim \nu\|\boldsymbol{R}\|_{L^{1}_{t,x}}^{1/p},\\
\|\boldsymbol{w}\|_{L^{s^{\prime}}_tL^{p^{\prime}}_x}&\lesssim\nu^{-1}\|\boldsymbol{R}\|_{L^{1}_{t,x}}^{1/p^{\prime}},\\
\|\boldsymbol{w}\|_{L^{\tilde{s}}_tW^{1,q}_x}&\le \delta,\quad\|\boldsymbol{R}^1\|_{L^{1}_{t,x}}\le \delta.
\end{align*}
The inessential constants in the estimates can be taken as a universal constant $M$, that is 
\begin{equation*}
\|\theta\|_{L^s_tL^p_x}\le\nu M\|\boldsymbol{R}\|_{L^{1}_{t,x}}^{1/p},\quad\|\boldsymbol{w}\|_{L^{s^{\prime}}_tL^{p^{\prime}}_x}\le\nu^{-1}M\|\boldsymbol{R}\|_{L^{1}_{t,x}}^{1/p^{\prime}}.
\end{equation*}
Notice \eqref{eq-10} follows from Lemma~\ref{lem4-7}, meanwhile \eqref{eq-11} follows from the  definition \eqref{eq-16} of $\theta$ and the fact $\operatorname{supp}\chi_j\subset I_{r/2}$.
 Proposition~\ref{prn3.1} follows. \qed

\section*{Acknowledgments}
This work is supported by the National Natural Science Foundation of China (Grant No.11871024).

\bibliographystyle{abbrv}
\bibliography{mybibfile}
\end{document}